\documentclass{amsart}
\usepackage{amsfonts}
\usepackage{amssymb}
\usepackage{graphicx}
\usepackage[active]{srcltx}
\usepackage{mathdots}
\usepackage{color}

\newtheorem{theorem}{Theorem}
\newtheorem{corollary}[theorem]{Corollary}
\newtheorem{lemma}[theorem]{Lemma}
\newtheorem{proposition}[theorem]{Proposition}
\theoremstyle{definition}
\newtheorem{remark}[theorem]{Remark}
\newtheorem{example}[theorem]{Example}
\newtheorem{definition}[theorem]{Definition}

  \DeclareMathOperator{\ii}{i} 
 \DeclareMathOperator{\const}{const} \DeclareMathOperator{\IM}{Im}
  
\DeclareMathOperator{\tr}{tr} \DeclareMathOperator{\RE}{Re}
\DeclareMathOperator{\diag}{diag} \DeclareMathOperator{\W}{W}
\DeclareMathOperator{\MP}{MP}
\newcommand{\x}{\star}
\newcommand{\Real}{\mathbb{R}}
\newcommand{\Comp}{\mathbb{C}}
\newcommand{\Nat}{\mathbb{N}}

\newcommand{\Ex}{\mathbb{E}}
\newcommand{\Prob}{\mathbb{P}}

\newcommand{\eps}{\varepsilon}
\newcommand{\dts}{,\dots,}

\newcommand{\sbs}{\subset}

\newcommand{\set}[1]{\left\{#1\right\}}

\newcommand{\norm}[1]{\left\Vert#1\right\Vert}
\newcommand{\maxnorm}[1]{\left\Vert#1\right\Vert_{\max}}

\newcommand{\matp}[1]{\begin{bmatrix} #1 \end{bmatrix}}

\providecommand{\norm}[1]{\left\Vert#1\right\Vert}

\usepackage{color}

\begin{document}

\renewcommand{\thefootnote}{\fnsymbol{footnote}}

   \title[  Random Perturbations of Matrix Polynomials ]
   { Random Perturbations of Matrix Polynomials}
   \author[P.Pagacz, M. Wojtylak]{ Patryk Pagacz \and Micha\l{} Wojtylak$^*$ }
\thanks{MW acknowledges the financial support by the NCN (National Science Center) grant, decision No. DEC-2013/11/B/ST1/03613}
\footnotetext[0]{ Faculty of Mathematics and Computer Science,
   Jagiellonian University,
   \L ojasiewicza 6
   30-348 Krak\'ow
  Poland, \texttt{patryk.pagacz@im.uj.edu.pl, michal.wojtylak@uj.edu.pl$^{*}$} }
\footnotetext[1]{corresponding author
}

\begin{abstract}
A sum of a large-dimensional random matrix polynomial and a fixed low-rank matrix polynomial is considered. The main assumption is that the resolvent of the random polynomial converges to some deterministic limit. A formula for the limit of the resolvent of the sum is derived and the eigenvalues are localised.
Four instances are considered: a low-rank matrix perturbed by the Wigner matrix, a product $HX$ of a fixed diagonal matrix $H$ and the Wigner matrix $X$ and two special matrix polynomials of higher degree. The results are illustrated with various examples and numerical simulations.
\end{abstract}
\thanks{}

\maketitle

\section*{Introduction}

\subsection*{Motivation}
Since the seminal works of Wigner \cite{wigner} and Marchenko and Pastur \cite{MP69}
spectral theory of random matrices has gathered a huge interest.  In particular, studying  the limit laws of eigenvalues was considered many times in the literature (\cite{Bai,BGnew,BGRnew,BCnew,Byrc,cap1,cap2,FP,Pizzo,Rnew,Tnew}).
 One of the recent techniques in this field is to investigate the limit (in probability) of the resolvent
$
(X_N-zI_N)^{-1}.
$
This was done already for Hermitian matrices, e.g. when $X_N\in\Comp^{N\times N}$ is a generalised Wigner matrix, see \cite{knowles2,Erdos2012,knowles3,knowles,knowles2017}.
In particular, the local isotropic semicircle law, stated in \cite{knowles2}, says that for a suitably chosen family of compact set $\mathbf{S}_N$ in the upper half-plane
 $$
 \sup_{z\in\mathbf S_N}\maxnorm{ (X_N -zI_N)^{-1} - m_{\x}(z) I_N }
 $$
 converges in probability, with a rate $\mathcal{O}(N^{-\frac\omega2})$, to zero, see Example \ref{Wigner} for details.  Here $m_{\W}(z)$ denotes the Stjelties transform of the Wigner semicircle law. As the sets $\mathbf{S}_N$ approach the real line, the local isotropic semicircle law becomes a tool to study the distribution of the eigenvalues.

Our aim is to investigate the limit of the resolvent for  some classes of nonsymmetric matrices and matrix polynomials. Let us recall that already several studies have addressed the canonical forms of nonrandom structured matrices and matrix polynomials  \cite{GLR} and their
change under a low-rank perturbation, see e.g. \cite{Batzkeetal,TeDo07,TeDo09,MoDo02,LM03,MMRR1,MMRR2,MMRR3,MMRR4}.
%
However,  the theory of random matrix polynomials is yet uncharted.

The  additional motivation for the current research  lies in noncommutative probability. Recall that 
 deforming a random matrix $X_N$ one obtains a deformation of the moment expansion of its limiting resolvent. This was already studied  in \cite{MWPP} for
$
X_N=H_NW_N,
$
where $W_N$ is a symmetric Wigner matrix and $H_N=\diag(c_1,1,\dots,1)$, $c_1\in\Real$. See also \cite{bozjwys98,bozjwys01}  for other works on moment deformations.

\subsection*{The results}
Let us recall first the basic notions.   For a matrix polynomial $X(z)=\sum_{j=0}^k X^{(j)} z^j$, with $X^{(j)}\in \Comp^{N\times N}$, $j=1,\dots n$ a point $\lambda\in\Comp$ is called an \emph{eigenvalue} if $X(\lambda)x=0$ for some nonzero $x\in\Comp^N$.
A polynomial is called \emph{regular} if $\det X(z)$ is a nonzero function. In such case, the matrix $X(\lambda)$ ($\lambda\in\Comp$) is invertible if and only if $\lambda$ is not an eigenvalue. This allows us to define the resolvent of a regular matrix polynomial as
$$
X(z)^{-1}:=\left( \sum_{j=0}^k z^j X^{(j)} \right)^{-1},
$$
which is a matrix valued rational function with   poles in the eigenvalues. We will consider eigenvalues and resolvent only for regular polynomials with the leading coefficient being invertible matrix. Hence, we will not  investigate the eigenvalue infinity. Let us turn now to the main results of the paper, a further review of necessary linear algebra and probability notions is contained in {\bf Section \ref{sP}}.

In {\bf Section \ref{sM}} we will consider a general setting of random matrix polynomials
$$
X_N(z)=\sum_{i=0}^k z^iX_N^{(i)}\in\Comp^{N\times N}[z],
$$
where the degree $k$ of the polynomial is fixed and does not depend on $N$ and  the matrices $X_N^{(i)}$ are either deterministic or random. The leading assumption is that the polynomias $X_N(z)$ are regular and $X_N(z)$ is  invertible on a set $S_N\subseteq\Comp$, $S_N\subseteq S_{N+1}$ for  $N=1,2,\dots$ and the  resolvent $X_N(z)^{-1}$ converges, pointwise in $z$, in probability to $M(z)$ on the union of $S_N$ (see {\bf Definition \ref{Defiso}} for details).  In such setting we will investigate
how these objects behave
under a low rank perturbation $X_N(z)+A_N(z)$.
 Our first main result, {\bf Theorem~\ref{thres}}, states precisely how the sets $S_N$, the limit $M_N(z)$ and the convergence rate are deformed in this general situation. Further, in
 {\bf Theorem \ref{spectrum}}, we locate and count the eigenvalues of  $X_N(z)+A_N(z)$, appearing in the union of $S_N$, after such deformation.

Further sections are devoted to the study of concrete ensembles. 
And so, {\bf Section~\ref{s3}} contains a result on low rank non-Hermitian perturbations of Wigner and random sample covariance matrices.  We  obtain the limit of the resolvent $(A_N+X_N-zI_N)^{-1}$ and show the limit points and convergence rate of the non-real points of the spectrum in {\bf Theorem \ref{matrixmain}}. Shortly these  can be formulated as follows.

{\it Let $X_N$ be a generalised Wigner or sample covariance matrix and let $A_N$  be deterministic, fixed low rank perturbation, e.g., $
A_N:= C\oplus 0_{N-n,N-n} \in\Comp^{N \times N}.
$
Then the resolvent  of $A_N+X_N$ converges in the maximum norm in probability to
$$
\widetilde{M}_N(z)=\big [ m_{\x}(z)I_N-m_{\x}^2(z)((C^{-1}+m_{\x}(z)I_n)^{-1}\oplus 0_{N-n,N-n})\big].
$$
Furthermore, if $z_0\in\Comp^+$ 
is such that $\xi=-\frac 1{m(z_0)}$ is an eigenvalue of
 $C$ with the algebraic multiplicity $k_\xi$ and the size of the largest Jordan block equal to $p_\xi$, then the $k_\xi$ eigenvalues $\lambda_{1}^N,\dots,\lambda^N_{k_\xi}$ of $X_N+A_N$ closest to $z_0$ are simple and converge to $z_0$ in probability, with the rate $\mathcal{O}( N^{-\frac1{ 2p_\xi}})$.}


 Matrices of type $HX$ with $H=H^*$ invertible  are well known in linear algebra, see e.g \cite{GLR,MMRR1}.
{\bf Section \ref{HW}} discusses products $H_NX_N$, where $H_N$ is a deterministic diagonal matrix with $(H_N-I_N)$ being of fixed  low rank  and $X_N$ is a Wigner or a random sample covariance matrix. In {\bf Theorem~\ref{MainForX_N2}} we provide the limit of the resolvent and limit and convergence rate of nonreal eigenvalues for $H_MX_N$. It is important to notice that already in this matrix problem it is necessary to apply the main results to nontrivial matrix polynomials of degree one (linear pencils). Namely, we set $X_N(z)=X_N-zI_N$, $A_N(z)=-zH_N^{-1}(I-H_N))$ so that $H_N(X_N(z)+A_N(z))=H_NX_N-zI_N$.

{\bf Section \ref{spoly}} contains a study of matrix polynomials of the form
$$
X_N- p(z)I_N + q(z)u_N u_N^*
$$
 and
 $$
  z^2(C_n\oplus0_{N-n}+I_N)+zX_N+D_n\oplus0_{N-n},
$$
where $p,q$ are polynomials, $X_N$ is either a Wigner or a random sample covariance matrix, $u_N$ is some deterministic vector and $C_n,D_n\in\Comp^{n\times n}$ are diagonal deterministic matrices.
This choice is motivated by the fact that matrix polynomials of this form appear in numerical methods for partial differential equations, see \cite{BHMST}. Again, we localise the spectrum of the given above polynomial by means of Theorems \ref{thres} and \ref{spectrum} and show difficulties appearing in a particular example.


\subsection*{Relation to existing results}

The main novelty of the current paper lies in the formula for the limit of the resolvent and in considering nonlinear eigenvalue problems. So far the limit laws for the resolvent were considered only for matrices and were of isotropic type, i.e., $m(z)I_N$ with a scalar analytic function $m(z)$.  Our construction leads to limit laws of different type and also for polynomials of degree greater than one.
Note that although eigenvalue problems for matrices are special cases of eigenvalue problems for polynomials,  many of the methods suitable for matrices, like, e.g., analysis of $\tr X^n$, cannot be adapted in the polynomial case. Our general theorems form Section \ref{sM}  develop a method which works like a `black box': knowing a limit of the resolvent of a polynomial $X_N(z)$ one is able to compute the limit of the resolvent of $X_N(z)+A_N(z)$.
Furthermore, by detecting the sets on which $X_N(z)+A_N(z)$ is invertible, we localise the eigenvalues.

As it was already said, our technique is applied also to matrices, i.e., to polynomials $X_N-zI_N$.
Although the low-rank perturbations of Wigner matrices were considered in many papers, see e.g.  \cite{BG1, BG2, BG3,BGnew,BGRnew,knowles2,BCnew,cap1,Rnew,Tnew}, the authors usually concentrate  on Hermitian or symmetric perturbations. The exceptions are the papers \cite{Knew,Rnew}, see the former for  the literature on physical motivations. In the latter paper Rochet considered the possibly non Hermitian  finite-rank perturbations $A$ of Wigner matrices $W$, proving   results  on the limit and convergence rate of non-real eigenvalues. More precisely,  the `Furthermore' part of Theorem \ref{matrixmain} (see also the simplified version above) is, generally speaking, a repetition of Theorems 2.3 and 2.10 from \cite{Rnew}.
Note that the paper \cite{Rnew} was a continuation of \cite{BGRnew}, where the authors considered a low-rank perturbation of a random matrix with a distribution  invariant under the left and right actions of the unitary group. Analogous  convergence rates for outliers $\mathcal{O}( N^{-1/({ 2p_\xi})})$  were obtained therein.

In the current paper we show how the resolvent tools can be used to find the convergence
rate of the eigenvalues of matrices converging to the non-real limits,
repeating the aforementioned result on eigenvalues from \cite{Rnew}.
However, in addition to \cite{Rnew}, we provide a formula and convergence rate for the limit of the resolvent after perturbation. We also estimate the rate of the convergence (to zero) of the imaginary part non-real  eigenvalues which are not outliers, see Example \ref{Wignerplusi2}. These two aspects were not studied in \cite{Rnew}.

The results on the products $H_NW_N$ also refine the existing ones  from \cite{MWPP, Wojtylak12b}
by showing the limit of the resolvent and considering a much wider class of $H_N$.

The last section on polynomials contains original, up to our knowledge, results on nonlinear eigenvalue problems with random coefficients.

\subsection*{The outcome}
There are two main outcomes of the present paper:
\begin{itemize}
\item extension of the knowledge of limit laws for the resolvents by providing new limit laws for the resolvents of polynomials of type  $X_N+A_N-zI_N$,  $X_N-zI_N-zA_N$, $p(z)I_N +q(z) A_N + X_N$, $
  z^2(A_N+I_N)+zX_N+B_N,
$where $A_N$ are $B_N$ are  low rank  and non-symmetric matrix and $p(z)$ and $q(z)$ are scalar polynomials, we stress that these limit laws are no longer isotropic;
\item analysis of limits in $N$ of spectra of  polynomials of the above type, with a special emphasis on investigating the convergence rates.
\end{itemize}

In the future, employing results for non-symmetric matrices or structured matrix polynomials would be most desirable, see e.g. \cite{RajanAbbot} for applications in neural networks. However,  the limit laws for the resolvent have not been discovered yet, see \cite{bordenave} for a review. Nonetheless, the general scheme we propose in Section \ref{sM} is perfectly suited for studying those as well.

\color{black}

\section{Preliminaries}\label{sP}
\subsection{Linear algebra}

First let us introduce various norms on spaces of matrices.
If $b$ is a vector then by $\norm{b}_p$ we denote the $\ell^p$-norm of $b$. If $A\in\Comp^{k \times l}$  then
$$
\norm{A}_{p,q}:=\sup_{x\neq 0}\frac{\norm{Ax}_q}{\norm x_p},\quad 1\leq p,q\leq\infty.
$$
We abbreviate $\norm{A}_{p,p}$ to $\norm{A}_p$. Recall that
\begin{equation}\label{norm1infty}
\norm{A}_{1,\infty}\leq \norm A_2
\end{equation}
\begin{equation}\label{norm2infty}
\norm{A}_{2,\infty}\leq \norm A_2,
\end{equation}
\begin{equation}\label{norm12}
\norm{A}_{1,2}\leq \norm A_2.
\end{equation}
Recall also the following formulas, valid  for $A=[a_{ij}]\in\Comp^{k\times l}$,
\begin{equation}\label{H1new}
\norm{A}_{1}=\max_{1\leq j\leq l}\sum_{i=1}^k |a_{ij}|,\quad  \norm{A}_{\infty}= \max  _{1 \leq i \leq k} \sum _{j=1} ^l | a_{ij} |.
\end{equation}
Further, $\norm A_{\max}$ denotes the maximum of the absolute values of all entries of $A$, and clearly
\begin{equation}\label{normmaxpq}
\maxnorm{A}\leq \norm A_{p,q}\leq k\maxnorm{A},\quad 1\leq p,q\leq \infty.
\end{equation}
By $I_N$ we denote the identity matrix of size $N$.
For  matrices $P\in\Comp^{N\times n}$ and $Q\in\Comp^{n\times N}$  we define
$$
\kappa_1(Q):=\sup_{E\in\Comp^{N\times N},\ E\neq 0}\frac{\norm{QE}_{1}}{\norm{E}_{\max}},
$$
$$
\kappa_\infty(P):=\sup_{E\in\Comp^{N\times N},\ E\neq 0}\frac{\norm{EP}_{\infty}}{\norm{E}_{\max}},
$$
\begin{equation}\label{defkappa}
\kappa(P,Q):=\sup_{E\in\Comp^{N\times N},\ E\neq 0}\frac{\norm{QEP}_{2}}{\norm{E}_{\max}}
\end{equation}
Let us denote the maximal number of nonzero entries in each row of $Q$ by $r(Q)$ and  the maximal number of nonzero entries in each column of $P$ by $c(P)$.
\begin{proposition}\label{ex1}
For  $Q\in\Comp^{n\times N}$ and $P\in\Comp^{N\times n}$ the following inequalities hold
$$
\kappa_1(Q)\leq n\cdot r(Q)\norm{Q}_{\max} ,
\quad
\kappa_\infty(P)\leq n\cdot c(P)\norm{P}_{\max} ,
$$ $$
 \kappa(P,Q)\leq n\cdot  r(Q)\  c(P)  \norm{Q}_{\max}\norm{P}_{\max}\leq n\cdot  r(Q)\  c(P)  \norm{Q}_2\norm{P}_2.
$$
\end{proposition}

\begin{proof}
For $Q=[q_{ij}]$, $E=[d_{ij}]$ we obtain,  using formula \eqref{H1new}, the following
\begin{eqnarray*}
\|QE\|_{1}& =  & \norm{  \left[ \sum_{ j=1 }^N q_{ij}d_{jk} \right]_{ik}}_1\\
&= &\max_{1\leq k\leq N} \sum_{i=1}^n \left|  \sum_{ j=1 }^N  q_{ij}d_{jk} \right| \\
& \leq  & \max_{1\leq k\leq N} \sum_{i=1}^{n} \sum_{j=1}^N |q_{ij}||d_{jk}|\\
& \leq &n r(Q) \norm{Q}_{\max} \|E\|_{\max}.
\end{eqnarray*}
Similarly,
\begin{eqnarray*}
\|EP\|_{\infty}& =  & \norm{  \left[ \sum_{ j=1 }^N d_{ij}p_{jk} \right]_{ik}}_\infty\\
&= &\max_{1\leq i\leq N} \sum_{k=1}^n \left|  \sum_{ j=1 }^N  d_{ij}p_{jk} \right| \\
& \leq  & \max_{1\leq k\leq N} \sum_{k=1}^{n} \sum_{j=1}^N |d_{ij}||p_{jk}|\\
& \leq &n c(P) \norm{P}_{\max} \|E\|_{\max}.
\end{eqnarray*}
The last claim results from the inequalities
\begin{eqnarray*}
\norm{QEP}_2&\leq & n\norm{QEP}_{\max}=n \max_{i,j=1\dts n}\left| \sum_{k,l} q_{ik}d_{kl}p_{lj}\right|\\
& \leq& n r(Q)c(P)\norm{E}_{\max} \norm{P}_{\max}\norm{Q}_{\max},
\end{eqnarray*}
and the relation \eqref{normmaxpq}.

\end{proof}

 The following  elementary result on matrices will be of frequent use. Let $A,B\in\Comp^{n\times n}$ and let $A$ be nonsingular.
Then $A+B$ is nonsingular if and only if $I_n+BA^{-1}$ is nonsingular, and in such case
\begin{equation}\label{ABinv}
 (A+B)^{-1}=A^{-1}(I_n+BA^{-1})^{-1}.
\end{equation}
Let $\norm\cdot$ denote any matrix norm. Then
\begin{equation}\label{ABnorm}
 \norm{(A+B)^{-1}-A^{-1}}\leq\norm{A^{-1}}^2\norm{(I_n+BA^{-1})^{-1}}\norm B.
\end{equation}
Furthermore,
\begin{equation}\label{ABtrue}
 \text{if }\norm{BA^{-1}}<1\text{ then }A+B\text{ is invertible }
 \end{equation}
 and
\begin{equation}\label{ABnorm2}
 \norm{(A+B)^{-1}-A^{-1}}\leq\frac{\norm{A^{-1}}^2\norm B}{{1-\norm{BA^{-1}}}}.
\end{equation}
In many places of this article we will use the well known Woodbury matrix identity. Let us recall that for invertible matrices $X\in\Comp^{N\times N}$, $C\in \Comp^{k \times k}$, and matrices $P\in \Comp^{N \times k}$, $Q\in \Comp^{k \times N}$, the matrix $X+PCQ$ is invertible if and only if $L:=C^{-1}+QX^{-1}P$ is invertible. In such case
\begin{equation}\label{Woodb}
(X+PCQ)^{-1}=X^{-1}-X^{-1}PL^{-1}QX^{-1}.
\end{equation}

\subsection{Probability theory}
In the whole paper we will work with one probability space, which is hidden in the background in the usual manner.
By $\Prob$ and $\Ex$ we denote the probability and expectation, respectively.  We will use the symbol `const' to denote a universal constant, independent from $N$.

\begin{definition}\label{stochdom}
Let
$$
\xi=\{\xi^{(N)}(u): N\in\Nat,\ u\in U^{(N)}\}, \quad \zeta=\{\zeta^{(N)}(u): N\in\Nat,\ u\in U^{(N)}\}
$$
be two families of nonnegative random variables, where $U^{(N)}$ is possibly an $N$-dependent parameter set. We say that $\xi$ is \textit{stochastically dominated by $\zeta$ 
simultaneously in $u$}, if for all $\eps>0$ and $\gamma>0$ we have
\begin{equation}\label{q2}
\Prob \left( \bigcap_{u\in U^{(N)}} \set{\xi^{(N)}(u)\leq N^\eps \zeta ^{(N)}(u)}\right)\geq 1- N^{-\gamma}
\end{equation}
for large enough $N\geq N_0(\eps,\gamma)$.
We will denote the above definition in symbols as  $\xi\prec \zeta$, usually remarking that 
the convergence is 
simultaneous and naming the set of parameters.

Furthermore, we say that $N$-dependent event $\Delta=\{\Delta^{(N)}(u)\subset \Omega :N\in\Nat,\ u\in U^{(N)}\}$ holds (\emph{simultaneously in $u$}) \emph{ with high probability} if $1$ is stochastically dominated by $\textbf{1}_\Delta$ simultaneously in $u$, equivalently,
if for all $\gamma>0$ we have
$$
\Prob \left( \bigcap_{u\in U^{(N)}} \Delta^{(N)}(u)\right) > 1-N^{-\gamma}
$$
for large enough $N\geq N_0(\gamma)$.

\end{definition}

\begin{remark}\label{whyN^-1}
In the sequel we will use without saying  the following facts:
\begin{itemize}
\item[] if $\xi^{(N)}(u)\leq\zeta^{(N)}(u)$ for all $u\in U^{(N)}$ and $N$ sufficiently large then $\xi\prec\zeta$,
\item[] if $\xi\prec\eta\prec\zeta$ then $\xi\prec\zeta$,
\item[] if $\xi_1,\dots,\xi_l\prec\zeta$, $\alpha_1,\dots,\alpha_n\geq0$ then $\alpha_1\xi_1+\dots+\alpha_l\xi_l\prec\zeta$,
\item[] if $\alpha>0$ and $\xi\prec N^{-\beta}$ for all $0<\beta<\alpha$ then $\xi\prec N^{-\alpha}$,
\end{itemize}
where $\xi,\eta,\zeta,\xi_1,\dots,\xi_l$ denote families of nonnegative random variables with a parameter set $U^{(N)}$, as in Definition \ref{stochdom}.
\end{remark}

\begin{remark} Recall that in  the literature (cf. e.g. \cite{knowles2} Definition 2.1)  the symbol $\prec$ denotes the {\it uniform} stochastic domination, namely
for all $\eps>0$ and $\gamma>0$ we have
\begin{equation}\label{q1}
\sup_{u\in U^{(N)}} \Prob \set{ \xi^{(N)}(u)> N^\eps \zeta ^{(N)}(u)}\leq N^{-\gamma}
\end{equation}
for large enough $N\geq N_0(\eps,\gamma)$. It is clear that simultaneous stochastic domination implies uniform stochastic domination and if the variables are Lipschitz continuous then the converse implication also holds, see, e.g., Remark 2.6 of \cite{knowles2} or Corollary 3.19 of \cite{erdos13}.  See also Lemma 3.2 of \cite{knowles2} for other properties of stochastic uniform domination.
To avoid assuming Lipschitz continuity, we will speak only about simultaneous stochastic domination.
\end{remark}

Let us now introduce one of the main objects of our study: a  limit law for the resolvent, defined here for random matrix polynomials.

\begin{definition}\label{Defiso}
Let
$$
X_N(z)=\sum_{i=0}^k z^iX_N^{(i)}\in\Comp^{N\times N}[z]
$$
be a random matrix polynomial, i.e.  the matrices $X_N^{(i)}$ are either deterministic or random matrices, and the degree $k$ of the polynomial is fixed and does not depend on $N$. Let $\mathbf{S}_N\sbs\Comp$ be a family of deterministic  open sets with $\mathbf{S}_N\sbs\mathbf{S}_{N+1}$ for all $N$ and let
$$
M_N:\mathbf{S}_N\to\Comp^{N\times N},\quad \Psi_N(z):\mathbf{S}_N\to[0,+\infty]
$$
be sequences of  deterministic functions such  that for any $z\in\bigcup_N\mathbf{S}_{N}$ the sequence $(\Psi_N(z))_{N}$  converges to zero.
We say that the resolvent $X_N(z)^{-1}$ \textit{has the limit  law $M_N(z)$ on sets $\mathbf{S}_N$ with the rate $\Psi_N(z)$ } if the eigenvalues of $X_N(z)$ are with high probability outside the set $\mathbf{S}_N$    and
$$
\norm{X_N(z)^{-1}-M_N(z)}_{\max} \prec \Psi_N(z)
$$
simultaneously in $z\in\mathbf{S}_N$.
\end{definition}

\begin{remark}
Note that the requirement that the eigenvalues are with high probability outside $\mathbf S_N$ should be read formally as:
 the (parameter-free) event $\{$the eigenvalues are outside ${\mathbf S}_N\}$ holds with probability $\geq 1-N^{-\gamma}$ for $N$ large enough and all $\gamma>0$. This is equivalent to saying that  the polynomial $X_N(z)+A_N(z)$ is invertible with high probability simultaneously in $z\in { \mathbf S}_N$.
\end{remark}

\color{black}

We present main examples, which are the motivation for the above definition.

\begin{example}\label{Wigner} Let $W=W_N=W_N^*$ be an $N\times N$ Hermitian matrix whose entries $W_{ij}$ are independent complex-valued random variables for $i\leq j$, such that
\begin{equation}\label{assW1}
\Ex W_{ij}=0,\quad \const\leq N\Ex|W_{ij}|^2,\quad \sum_j \Ex|W_{ij}|^2=1,
\end{equation}
and for any $p\in\Nat$ expectation of $|\sqrt{N}W_{ij}|^p$ is bounded, i.e.
\begin{equation}\label{assW2}
\Ex|\sqrt{N}W_{ij}|^p\leq \const(p),
\end{equation}
where $\const(p)$ denotes a constant depending on $p$ only.

The function $$m_{\W}(z)= \frac{-z+\sqrt{z^2-4}}{2}$$ is the Stieltjes transform of Wigner semicircle distribution.
It was shown in \cite{knowles2} (see also \cite{knowles3}) that for each $\omega\in(0,1)$, the resolvent $(W_N-zI_N)^{-1}$ has a  limit law $M_N(z)=m_{\W}(z)I_N$ on the set
$$
\mathbf{S}^{\W}_{N,\omega}=\set{z=x+\ii y: |x|\leq \omega^{-1},\ N^{-1+\omega}\leq y\leq\omega^{-1}},
$$
with the rate
$$
\Psi^{\W}_N(z)=\sqrt{\frac{\IM m_{\W}(z)}{Ny}}+\frac1{Ny}.
$$
Indeed, this can easily be deduced from Theorem 2.12, remark after Theorem 2.15 (see also Remark 2.6) and Lemma 3.2(i) of \cite{knowles2}.
The authors call this the isotropic local limit law  because of the form $M_N(z)=m_{\W}(z)I_N$. In the next section we will provide polynomials with the resolvent having  limit law
of a different type.
Furthermore, since $|m(z)|\leq \omega^{-1}$ for $z\in \mathbf{S}_N$,
one has
\begin{equation}\label{ratesup1}
\sup_{z\in \mathbf{S}_{N,\omega}^{\W}}|\Psi_N^{\W}(z)| \leq \Big(\sqrt{\frac{\omega^{-1}}{N N^{-1+\omega}}}+\frac{1}{N N^{-1+\omega}}\Big)=\mathcal{O}(  N^{-\frac\omega2}).
\end{equation}

Another example of a resolvent having  a limit law is given by the same polynomial $W_N-zI_N$ but now with
$\mathbf{S}_N=\mathbf{T}$, where $\mathbf{T}$ is some compact set in the upper half-plane. Observe that in this setting we again have  $M_N(z)=m_{\W}(z)I_N$ with the same rate $\Psi_N^W(z)$, but the estimate \eqref{ratesup1} can be improved to
\begin{equation}\label{ratesup2}
\sup_{z\in\mathbf T}  \Psi^{\W}_N(z)=\mathcal{O}(N^{-\frac{1}2}).
\end{equation}
In what follows we will need both constructions presented in this example.

\end{example}

Our second example  is the isotropic local Marchenko-Pastur limit law.

\begin{example}\label{MP}
Let $Y=Y_N$ be an $M\times N$ matrix, with $N,M$ satisfying
\begin{equation}\label{assMP1}
N^{1/\const}\leq M\leq N^{\const}
\end{equation}
whose entries $Y_{ij}$ are independent complex-valued random variables such that
\begin{equation}\label{assMP2}
\Ex Y_{ij}=0,\quad \Ex|Y_{ij}|^2=\frac1{\sqrt{NM}},
\end{equation}
and for all $p\in\Nat$
\begin{equation}\label{assMP3}
 \Ex|(NM)^{1/4}Y_{ij}|^p=\const(p).
\end{equation}
Let also
$$
\phi=M/N,\quad \gamma_\pm=\sqrt\phi+\frac1{\sqrt\phi}\pm 2,
$$ $$
 \kappa(\RE z)=\min(|\gamma_-- \RE z|,|\gamma_+-\RE z|), \quad K=\min(N,M).
$$
Then resolvent of the polynomial $Y_N^*Y_N-zI_N$ has a  limit law $M_N=m_{\MP}(z) I_N$, where
$$
m_{\MP}(z)=m_{\MP}^\phi(z) = \frac{\phi^{1/2}-\phi^{-1/2}-z+\ii\sqrt{(z-\gamma_-)(\gamma_+-z)}}{2\phi^{-1/2}z}
$$
on the set
$$
\mathbf{S}^{\MP}_{N,\omega}=\set{z=x+\ii y\in\Comp: \kappa(x)\leq\omega^{-1},\ K^{-1+\omega}\leq y \leq \omega^{-1}, \ |z|\geq \omega},
$$
with the rate
$$
\Psi^{\MP}_N(z)=\sqrt{\frac{\IM m_\phi(z)}{Ny}}+\frac1{Ny}.
$$
As in the previous example, this is an example of an isotropic limit law and can be deduced from the results in \cite{knowles2}: Theorem 2.4, Remark 2.6 and Lemma 3.2(i). Furthermore, one has that
\begin{equation}\label{ratesup1MP}
\sup_{z\in \mathbf{S}_{N,\omega}^{\MP}}|\Psi_N^{\MP}(z)| \leq \mathcal{O}(  N^{-\frac\omega2}).
\end{equation}
As in Example \ref{Wigner} we  change the setting by putting
$\mathbf{S}_N=\mathbf{T}$, where $\mathbf{T}$ is some compact set in the upper half-plane, which leads to the estimate
\begin{equation}\label{ratesup2MP}
\sup_{z\in\mathbf T}  \Psi^{\MP}_N(z)=\mathcal{O}(N^{-\frac{1}2}).
\end{equation}
\end{example}

Further examples of local limit laws in the literature (which are, in particular, limit laws for the resolvent) concern matrices of type $(Y_N-w I_N)^*(Y_N-w I_N)$, where $w$ is a complex parameter, applied in the Hermitisation technique, see \cite[Theorem 6.1]{bourgade}.

\begin{remark} It is worth mentioning, that having a limit law for the resolvent of the family of polynomials $X_N(z)=\sum_{j=0}^k z^jX_N^{(j)}$ it is relatively easy to derive a limit law for the resolvent of the polynomials $\alpha X_N(z) +\beta I_N$, $\alpha,\beta\in\Comp$, $zX_N(z)$ and $\text{rev}X_N(z):=\sum_{j=0}^k z^{k-j}X_N^{(j)}$. In particular, $X_N(z)=zX_N-I_N$, where $X_N$ is a generalised Wigner matrix, is an example of a first order polynomial having a limit law for the resolvent, with a nontrivial leading coefficient.

\end{remark}

%
%
%
%

We conclude this section with an example of a random matrix without a resolvent limit law, to show the difference between the resolvent limit law and stochastic convergence of  eigenvalues.

\begin{example}\label{Normal} Let $X_N\in\Comp^{N\times N}$ be a diagonal matrix, with elements on the diagonal being i.i.d. standard normal variables.
Although the empirical measures of the eigenvalues of $X_N$ converge weakly in probability to the normal distribution, the resolvent  $(X_N-zI_N)^{-1}$ does not converge in any reasonable sense.
\end{example}

\section{Main results}\label{sM}

\subsection{The resolvent}

In this subsection we will show how a low-dimensional perturbation deforms a resolvent limit law.
Recall that $r(B),c(B)$ denote, respectively, the maximal number of nonzero entries in each row and column of a matrix $B$.
\begin{theorem}\label{thres}
Let $(n_N)_N$ be a nondecreasing sequence  and let
$$
C_N(z)\in\Comp^{n_N\times n_N}[z],\quad A_N(z):= P_NC_N(z  )Q_N\in\Comp^{N\times N}[z],
$$
be deterministic  matrix polynomials, where $P_N\in\Comp^{N\times n_N}$, $Q_N\in\Comp^{n_N\times N}$.
Let  $X_N(z)\in\Comp^{N\times N}[z]$ be a random matrix  polynomial. We assume that
\begin{enumerate}
\item[(a1)] $X_N(z)^{-1}$ has a  limit law $M_N(z)$ on a family of sets $\mathbf{S}_N$ with rate $\Psi_N(z)$, see Definition \ref{Defiso},
\item[(a2)] $C_N(z)$ is invertible for $z\in\mathbf S_N$,
\item[(a3)] we have that
\begin{equation}\label{add0}
n_N \sup_{z\in\mathbf{S}_N}\Psi_N(z) \leq\mathcal{O}(N^{-\alpha}),
\end{equation}
for some $\alpha>0$,
\item[(a4)]  $\norm{P_N}_2,\norm{Q_N}_2,c(P_N),r(Q_N)\leq \const$.

\end{enumerate}

 Then, for any $\beta\in(0,\alpha)$,
 the eigenvalues of the random polynomial
 $X_N(z)+A_N(z)$ are with high probability outside the set
 \begin{equation}
\widetilde{\mathbf{S}}_N:=\bigg\{z\in\mathbf{S}_N:K_N(z)\textrm{ \rm is invertible, } \label{ZZZ}\\
 {\norm{K_N(z)^{-1}}_2}<N^{\beta}\bigg\},
\end{equation}
where
\begin{equation}\label{KNdef}
K_N(z)=C_N(z)^{-1}+Q_NM_N(z)P_N.
\end{equation}
Furthermore, the resolvent of
$X_N(z)+A_N(z)$ has a  limit law on $\widetilde{\mathbf{S}}_N$
\begin{equation}\label{resX}
\widetilde{M}_N(z)=M_N(z)-M_N(z)P_NK_N(z)^{-1}Q_NM_N(z),
\end{equation}
with the rate
$$
\widetilde{\Psi}_N(z):= N^\alpha n_N \Psi_N(z)\norm{M_N(z)}_2^2
$$
under the additional assumption that $\widetilde{\Psi}_N(z)$ converges to zero for $z\in\bigcup_N\mathbf{S}_N$.
\end{theorem}

\begin{proof} 
We  set $\beta=\alpha/2$, the proof for arbitrary $\beta<\alpha$ requires only few technical adjustments.
Fix arbitrary  $\gamma>0$.  Due to (a1) and  the definition of stochastic simultaneous domination  (Definition \ref{stochdom},  $\eps=\alpha/4$) we have that with $E_N(z)=X_N(z)^{-1}-M_N(z)$ the following event
\begin{equation}\label{Emax}
\Theta:=\set{\forall z\in\mathbf S_N \  \norm{E_N(z)}_{\max}\leq N^{\alpha/4}\Psi_N(z)},
\end{equation}
holds with probability $\geq 1-N^{-\gamma}$, for $N\geq N_0(\alpha,\gamma)$ sufficiently large. Note that, if $\Theta$ occurs, one has that, for all $z\in \widetilde{\mathbf{S}}_N$,
\begin{eqnarray*}
\norm{Q_NE_N(z)P_NK_N(z)^{-1}}_2&\leq& \norm{Q_NE_N(z)P_N}_2 N^{\frac\alpha2},\text{ by }\eqref{ZZZ},\\
&\leq &  \kappa(P_N,Q_N)\norm{E_N(z)}_{\max} N^{\frac\alpha2}, \text{ by \eqref{defkappa} }, \\
&\leq&  \kappa(P_N,Q_N)  N^{\frac\alpha2 +\frac\alpha4}\Psi_N(z), \text{ by \eqref{Emax}},\\
&\leq&  n_N c(P_N) r(Q_N) \norm{P_N}_2 \norm{Q_N}_2 N^{\frac{3\alpha}4} \Psi_N(z) , \text{ by Prop. \ref{ex1}}, \\
&\leq&  \const  \cdot n_N  N^{\frac{3\alpha}4}\sup_{z\in\mathbf{S}_N }\Psi_N(z),  \text{ by (a4)}.
\\
\end{eqnarray*}

Note that by the assumption (a3),  if $\Theta$ occurs then
$$
\forall z\in\widetilde{\mathbf S}_N \  \norm{Q_NE_N(z)P_NK_N(z)^{-1}}_2<1.
$$
Consequently, the above inequality holds with probability $\geq 1-N^{-\gamma}$, for sufficiently large $N\geq N_0(\alpha,\gamma)$. By the Woodbury matrix equality,  the matrix
$X_N(z)+P_NC_N(z)Q_N$ is invertible if and only if $C_N(z)^{-1}+Q_NX_N(z)^{-1}P_N$ is invertible.
Note that
$$
C_N(z)^{-1}+Q_NX_N(z)^{-1}P_N=Q_NE_N(z)P_N+K_N(z).
$$
This, together with \eqref{ABinv} implies that on the event $\Theta$ the matrix
$X_N(z)+P_NC_N(z)Q_N$ is invertible for sufficiently large $N\geq N_0(\alpha,\gamma)$. As $\gamma$ was arbitrary we see that  the eigenvalues of $X_N(z)+A_N(z)$ are outside $\widetilde{\mathbf{S}}_N$  with high probability.

Now we prove the convergence of $(X_N(z)+A_N(z))^{-1}$. Let
$$
E^{(1)}_N(z):=(C_N(z)^{-1}+Q_N(M_N(z)+E_N(z))P_N)^{-1} - K_N(z)^{-1}.
$$
Consider
$$
(X_N(z)+ A_N(z))^{-1}=(X_N(z)+P_NC_N(z)Q_N)^{-1}=
$$
$$
=X_N(z)^{-1}-X_N(z)^{-1}P_N(C_N(z)^{-1}+Q_NX_N(z)^{-1}P_N)^{-1}Q_NX_N(z)^{-1}
$$
$$
=(M_N(z)+E_N(z))-(M_N(z)+E_N(z))P_N(K_N(z)^{-1}+(E^{(1)})_N(z))Q_N(M_N(z)+E_N(z)).
$$
Confronting with \eqref{resX} and dropping the $z$-dependence  ($z\in\widetilde{\mathbf S}_N$) and the $N$-dependence we obtain the difference to estimate
\begin{eqnarray*}
E^{(2)}:&=&(X+A)^{-1} - \widetilde{M} \\
& = & E+MP E^{(1)}QM+EPK^{-1}QM+EPE^{(1)}QM\\
&+& MPK^{-1}QE+MP E^{(1)}QE+EPK^{-1}QE+EPE^{(1)}QE.
\end{eqnarray*}
We will estimate the maximum norm of each summand in the right hand side of the above equation. For this aim we state some preliminary inequalities. Recall that by Proposition \ref{ex1}, assumptions on $P_N$ and $Q_N$ and \eqref{normmaxpq} one has
 $$
 \kappa_1(Q_N),\kappa_\infty(P_N),\kappa(P_N,Q_N)\leq \const n_N.
 $$
The stochastic domination  below in this proof is simultaneous in  $z\in\widetilde{\mathbf{S}}_N$. One has
\begin{eqnarray}
	\norm{E_N(z)P_N}_\infty &\leq&\kappa_\infty(P_N)\norm{E_N(z)}_{\max}\prec n\Psi_N(z),\label{ep}
	\end{eqnarray}
	\begin{eqnarray}
	\norm{Q_NE_N(z)}_1 &\leq&\kappa_1(Q_N)\norm{E_N(z)}_{\max}\prec n_N\Psi_N(z),\label{qe}
	\end{eqnarray}
    \begin{eqnarray}
    \norm{E^{(1)}_N(z)}_2&=&\|(C_N(z)^{-1}+Q_N(M_N(z)+E_N(z))P_N)^{-1} \nonumber\\
    &\ & -  (C_N^{-1}(z)+Q_NM_N(z)P_N)^{-1} \|_2 \nonumber\\
&\leq &\frac{\norm{(C_N(z)^{-1}+Q_NM_N(z)P_N)^{-1}}_2^2 \norm{Q_NE_N(z)P_N}_2}{1-\norm{Q_NE_N(z)P_N}_2\norm{(C_N(z)^{-1}+Q_NM_N(z)P_N)^{-1}}_2}\nonumber, \quad \text{by \eqref{ABnorm2}}\\
&\leq& \frac{N^{\alpha} \norm{ E_N(z)}_{\max}\kappa(P_N,Q_N)}{1- \norm{ E_N(z)}_{\max}\kappa(P_N,Q_N)N^{\frac{\alpha}{2}}}\nonumber\\
&\prec&\frac{n_N\Psi_N(z)N^{\alpha}}{1-n_N \Psi_N(z) N^{\frac{\alpha}{2}}}\nonumber\\
&\prec& n_N\Psi_N(z)N^{\alpha}.
\label{e1}
\end{eqnarray}

We can  now derive the announced estimation of summands of $E^{(2)}(z)$. In the following estimations we again drop the $z$-dependence and the $N$-dependence, the  stochastic domination  below in this proof is simultaneous in  $z\in\widetilde{\mathbf{S}}_N$. And so we have
\begin{eqnarray*}
\norm{E}_{\max}&\prec &\Psi(z),
\end{eqnarray*}
	\begin{eqnarray*}
\maxnorm{MP E^{(1)}QM}&\leq &\norm{MP E^{(1)}QM}_{2}\\
& \leq & \norm{M}_2^2\norm P_2\norm Q_2 \norm{E^{(1)}}_2\\
&  \prec& n\Psi(z)N^{\alpha}\norm{M}_2^2, \quad \text{by \eqref{e1}}, \\
\end{eqnarray*}
\begin{eqnarray*}
\maxnorm{EPK^{-1}QM} &\leq& \norm{EPK^{-1}QM}_{2,\infty}\\
&\leq& \norm{EP}_\infty  \norm{ K^{-1}}_{2,\infty} \norm{QM}_2 \quad \text{by \eqref{ep} }, \\
&\leq& \kappa_\infty(P)\norm{E}_{\max}\norm{K^{-1}}_2\norm{Q}_2\norm{M}_2 \quad \text{by \eqref{norm2infty} }, \\
&\prec & n \Psi (z)  N^{\frac{\alpha}{2}} \norm{M}_{2}, \quad \text{by \eqref{ep}}, \\
\end{eqnarray*}
\begin{eqnarray*}
\maxnorm{EPE^{(1)}QM}
 &\leq& \norm{EPE^{(1)}QM}_{2,\infty}\\
&\leq& \norm{EP}_\infty  \norm {E^{(1)}}_{2,\infty} \norm{QM}_2\\
&\leq&\kappa_\infty(P)\norm{E}_\infty \norm{E^{(1)}}_{2}\norm{Q}_2\norm{M}_2 ,\quad\text{by \eqref{norm2infty},\eqref{e1}},\\
&\prec& n^2\Psi^2(z)N^{\alpha}\norm{M}_2\\
&\prec & N^{-\alpha}\norm{M}_2, \quad\text{ by \eqref{add0}}, \\
\end{eqnarray*}
	\begin{eqnarray*}
	\maxnorm{MPK^{-1}QE} &\leq& \norm{MPK^{-1}QE}_{1,2}\\
&\leq& \norm{MP}_2  \norm {K^{-1}}_{1,2} \norm{QE}_1,\quad \text{by \eqref{qe} }, \\
&\prec& \norm{M}_2\norm{P}_2   N^{\frac{\alpha}{2}} \kappa_1(Q)\Psi(z),\quad\text{by \eqref{norm12}},\\
  &\prec& n\Psi(z)N^{\frac{\alpha}{2}} \norm{M}_2,
\end{eqnarray*}
\begin{eqnarray*}
\maxnorm{MPE^{(1)}QE} & \leq & \norm{MPE^{(1)}QE}_{1,2}\\
&\leq& \norm{MP}_2\norm{E^{(1)}}_{1,2}\norm{QE}_1,\quad \text{by \eqref{qe}},\\
&\prec& \norm{M}_2  n\Psi(z)N^\alpha \kappa_1(Q)\Psi(z),\quad\text{by \eqref{norm12},\eqref{e1}},\\
&\prec&  n^2\Psi^2(z)N^{\alpha}\norm{M}_2 \\
&\prec & N^{-\alpha}\norm{M}_2,\quad\text{ by \eqref{add0}}, \\
\end{eqnarray*}
\begin{eqnarray*}
\maxnorm{EPK^{-1}QE} &\leq&  \norm{EPK^{-1}QE}_{1,\infty}\\
&\leq &\norm{EP}_\infty\norm{K^{-1}}_{1,\infty}\norm{QE}_1,\quad \text{by \eqref{qe} },\\
&\prec&\Psi(z) \kappa_\infty(P) N^{\frac{\alpha}{2}} \kappa_1(Q)\Psi(z),\quad \text{by \eqref{norm1infty}},\\
&\prec& n^2\Psi^2(z)N^{\frac{\alpha}{2}} \\
&\prec & N^{-3\alpha/2},\quad\text{ by \eqref{add0}}, \\
\end{eqnarray*}
\begin{eqnarray*}
\maxnorm{EPE^{(1)}QE } & \leq &   \norm{EPE^{(1)}QE}_{1,\infty}\\
&\leq& \norm{EP}_\infty\norm{E^{(1)}}_{1,\infty}\norm{QE}_1,\quad \text{by \eqref{qe} },\\
&\prec & \Psi(z)\kappa_\infty(P)n\Psi(z)N^{\alpha}\kappa_1(Q)\Psi(z),\quad \text{by \eqref{norm1infty},\eqref{e1}},\\\
&\prec & n^3\Psi^3(z)N^{\alpha}\\
&\prec & N^{-2\alpha},\quad\text{ by \eqref{add0}}.\\
\end{eqnarray*}

Due to the fact that
$$
N^{\frac{\alpha}{2}} \leq \const N^\alpha,\quad \norm{M_N}_2\leq \const \norm{M_N}_2^2, \quad n_N\Psi_N(z)N^{\frac{\alpha}{2}}\to 0
$$
the proof is finished.
\end{proof}

\subsection{The spectrum}\label{spectrumSection}

In the current subsection the dimension $n_N$ will be constant and denoted by $n$.
First let us prove a technical lemma.

\begin{lemma}\label{ab}
Let the matrices $A,B\in \Comp^{k\times k}$.
Then
 $$
 \big| \det A - \det B \big|\leq k!\cdot k \maxnorm{A-B}\left(\maxnorm{A-B}+\maxnorm{A}\right)^{k-1}.
 $$
\end{lemma}
\begin{proof}
Observe that  with
$$
M=\max(\maxnorm A,\maxnorm B) \leq \maxnorm{A}+\maxnorm{A-B}
$$
 one has
\begin{eqnarray*}
\big| \det A - \det B \big| &\leq& \sum_{\sigma\in S_k} \big|a_{1\sigma(1)}a_{2\sigma(2)}\dots a_{k\sigma(k)}-b_{1\sigma(1)}b_{2\sigma(2)}\dots b_{k\sigma(k)}\big|
\\
&\leq& \sum_{\sigma\in S_k} \Big( |a_{1\sigma(1)}-b_{1\sigma(1)}||a_{2\sigma(2)}\dots a_{k\sigma(k)}|\\
&+&|b_{1\sigma(1)}||a_{2\sigma(2)}-b_{2\sigma(2)}||a_{3\sigma(3)}\dots a_{k\sigma(k)}|+\dots
\\
&+&|b_{1\sigma(1)}b_{2\sigma(2)}\dots b_{(k-1)\sigma(k-1)}||a_{k\sigma(k)}-b_{k\sigma(k)}| \Big)
\\
&\leq& k!\cdot k\maxnorm{A-B}M^{k-1}\\
&\leq & k!\cdot k\maxnorm{A-B} \left(\maxnorm{A-B}+\maxnorm{A}\right)^{k-1}.
\end{eqnarray*}
\end{proof}

The next step in the analysis of the spectra of matrices $X_N+A_N$ is the following theorem, for its formulation let us introduce a usual technical definition.
\begin{definition}
 Suppose a point $z_0\in\Comp$ is given. We order the complex plane with respect to the lexicographic order on $[0,+\infty)\times[0,2\pi)$ identifying a point $ \lambda$ with the pair $(|\lambda-z_0|, \arg(\lambda-z_0))$.
\end{definition}

\begin{theorem}\label{spectrum}
Let $n$ be fixed and let
$$
C(z)\in\Comp^{n\times n}[z],\quad A_N(z):= P_NC(z  )Q_N\in\Comp^{N\times N}[z],
$$
be deterministic  matrix polynomials, where  $P_N\in\Comp^{N\times n}$, $Q_N\in\Comp^{n\times N}$, $N=1,2,\dots$
Let $X_N(z)\in\Comp^{N\times N}[z]$ be a random matrix  polynomial. We assume that
\begin{enumerate}
\item[(a1)] $X_N(z)^{-1}$ has a  limit  law $M_N(z)$ on a family of  sets $\mathbf{S}_N$ with the rate $\Psi_N(z)$, see Definition \ref{Defiso},
\item[(a2.1)] $C(z)$ is invertible for $z\in\bigcup_N\mathbf{S}_N$,
\item[(a3.1)] the following estimate holds
$$
\sup_{z\in \mathbf{S}_N}|\Psi_N(z)|\leq \mathcal{O}(N^{-\alpha})$$
 with some $\alpha>0$,
 \item[(a4)] $\norm{P_N}_2,\norm{Q_N}_2,c(P_N),r(Q_N)\leq\const$,
\item[(a5.1)] the matrix-valued function $z\mapsto Q_NM_N(z)P_N$ is analytic on the interior of $\bigcup_N\mathbf{S}_N$ and does not depend on $N$.
\end{enumerate}
Let also
$$
K(z):=C(z)^{-1}+Q_NM_N(z)P_N,\quad {L_N}(z):=C(z)^{-1}+Q_NX_N(z)^{-1}P_N.
$$
Assume that the function $\det K(z)$ has a zero of order  $k>0$ at a point $z_0$ lying in the interior of $\bigcup_N\mathbf{S}_N$ and let  $\lambda_1^N,\dots,\lambda_k^N,\dots$ be the zeros of $\det {L_N}(z)$ written down with multiplicities in the order given by their distance to $z_0$. Then the first $k$ of them  converge to $z_0$ in the following sense
\begin{equation}\label{c1}
|\lambda_j^N-z_0|\prec N^{-\frac\alpha k},\quad j=1,2,\dots,k,
\end{equation}
 while the $k+1$-st, and in consequence all following ones, do not converge to $z_0$, more precisely, for any  $\beta>0$ the set of random variables $|\lambda_{k+1}^N-z_0|$ is not simultaneously stochastically dominated by $N^{-\beta}$.
\end{theorem}

\begin{proof}
Fix $\eps,\gamma>0$, with $\eps<\alpha$. We show that there are exactly $k$ zeros $\lambda_1^N,\dots,\lambda_k ^N$   of $\det {L_N}(z)$ in $B(z_0,N^{-\beta})$
with probability $\geq 1-N^{-\gamma}$ for $N\geq N_0(\eps,\gamma)$ large enough  and any $\beta\leq\frac{-\eps+\alpha}k$.
This will prove both statements. Indeed,  setting $\beta=\frac{-\eps+\alpha}k$ shows that condition \eqref{q2} in the definition of stochastic simultaneous domination
is satisfied for any $\eps<\alpha$, and hence in an obvious way for any $\eps>0$.
Setting $\beta$ to be arbitrary small shows that $|\lambda_{k+1}^N-z_0|$ is not simultaneously stochastically dominated by $N^{-\beta}$.

Let us fix an open bounded set $\mathbf T$ such that $z_0\in\mathbf T$ and the closure of $\mathbf T$ is contained in the interior of some $\mathbf S_{N}$ ($N\geq 1$).
We may assume without loss of generality that $X_N(z)$ is invertible on $\mathbf{T}$.
Note that due to (a2') and (a5) the function $K(z)$ is continuous on the closure of  $\mathbf{T}$, hence,
 $$
\sup_{z\in\mathbf{T}}\maxnorm{K(z)}\leq \const.
$$
Moreover, due to Proposition \ref{ex1} one has
\begin{eqnarray*}
\maxnorm{{L_N}(z)-K(z)} & \leq & \norm{ Q_N (X_N(z)^{-1}-M_N(z)) P_N}_2 \\
&\leq & n \norm{P_N}_2\norm{Q_N}_2 c(P_N) r(Q_N) \|X_N(z)^{-1}-M_N(z)\|_{\max}.
\end{eqnarray*}
Hence,   by (a1) and (a4) the probability of  the event
\begin{equation}\label{bbb2}
\set{ \forall_{z\in\mathbf{S}_N}\ \maxnorm{{L_N}(z)-K(z)}\leq N^{\frac \eps5}\Psi_N(z)}
\end{equation}
is higher than $1-N^{-\gamma},$ for $N\geq N_1(\eps,\gamma)$ large enough.
Note that $  N^{\frac \eps5}\Psi_N(z)\leq N^{\frac \eps4-\alpha} $ by (a3).
By Lemma \ref{ab} one has
$$
|\det {L_N}(z)-\det K(z)| \leq n\cdot n! \maxnorm{{L_N}(z)-K(z)} \left(\maxnorm{K(z)}+ \maxnorm{{L_N}(z)-K(z)}         \right)^{n-1}.
$$
Hence, the probability of the  event
\begin{equation}\label{bbb25}
\set{\forall_{z\in\mathbf T} \ |\det {L_N}(z)-\det K(z)| \leq n!n   N^{\frac\eps4-\alpha} \left(  N^{\frac\eps4-\alpha} +\maxnorm{K(z)}  \right)^{n-1}}
\end{equation}
is higher than $1-N^{-\gamma},$ for $N\geq N_2(\eps,\gamma)$ large enough.
As $\frac\eps3-\alpha<0$ and $\maxnorm{K(z)}$ is bounded on $\mathbf{T}$  the probability of the  event
\begin{equation}\label{bbb3}
\set{\forall_{z\in\mathbf T}\ |\det {L_N}(z)-\det K(z)| \leq  N^{\frac\eps3-\alpha}}
\end{equation}
is higher than $1-N^{-\gamma},$ for $N\geq N_3(\eps,\gamma)$ large enough.

Observe that
$$\det K(z) = (\det K)^{(k)}(z_{0})(z_{0}-z)^{k}+ o(|z_{0}-z|^{k}).$$
Consequently,
\begin{equation} \label{lhs}
|\det K(z)| \geq |(\det K)^{(k)}(z_{0})| N^{\frac\eps2-\alpha}>N^{\frac\eps3-\alpha},\quad z\in \partial B(z_{0},N^{-\beta}),
\end{equation}
for  sufficiently large $N\geq N_4(\eps,\gamma)$ and any $\beta\leq\frac{-\eps+\alpha}k$.

Combining \eqref{bbb3} and \eqref{lhs} we get
\begin{equation}\label{rouche1}
|\det {L_N}(z)-\det K(z)|<|\det K(z)|,\quad z\in \partial B(z_{0},N^{-\beta})
\end{equation}
 with probability higher than $1-N^{-\gamma}$ for $N\geq N_5(\eps,\gamma)$ large enough.

However,  \eqref{rouche1} implies, via the Rouch\'e theorem, that $\det K(z)$ and $\det {L_N}(z)$ have the same number of zeros in $B(z_0,N^{-\beta})$. Hence, there are exactly $k$ zeros $\lambda_1,\dots,\lambda_k $   of $\det {L_N}(z)$ in $B(z_0,N^{-\beta})$ with probability higher than $1-N^{-\gamma}$ for $N\geq N_5(\eps,\gamma)$ large enough.

\end{proof}

\begin{remark}
 Let us compare Theorems \ref{thres} and \ref{spectrum}. First note that the latter one has slightly stronger assumptions, which are however necessary for defining the limit point $z_0$, also the sequence $n_N$ is required there to be constant. Comparing the claims let us note that both theorems state more or less the same fact: the eigenvalues converge to some limit points with a certain convergence rate. If $n=1$ the claim of Theorem \ref{spectrum} is slightly than the one of Theorem \ref{thres}. Namely, the function $K(z)$ is a scalar function in this situation, and the condition
 $$
 |K(z)|\geq N^{-\alpha},
 $$
 which constitutes the set $\widetilde{ \mathbf{S}}_N$, locates with high probability the eigenvalues in the (approximate) discs around the points $z_0$ and with radius equal to $N^{-\beta}$, for $\beta<\alpha$, while  Theorem \ref{spectrum} already states that $\beta=\alpha$.

However, already for $n=2$ the estimates given by Theorem \ref{spectrum} are  weaker, we will see this more clearly in  Section \ref{s3}. The main reason for stating Theorem \ref{spectrum}, despite it is giving a weaker estimate, is that it allows us in some situations to count the eigenvalues, while Theorem \ref{thres} does not even guarantee that inside each connected component of the complement of $\widetilde{\mathbf{S}}_N$  there is any eigenvalue of $X_N+A_N$. Therefore, in what follows, we will use Theorem \ref{thres} to get the optimal convergence rate and Theorem \ref{spectrum} to get the number of eigenvalues which converge to $z_0$.
\end{remark}

\section{Random perturbations of matrices}\label{s3}
In this section we will consider the situation where  $A_N(z)=A_N$ is a matrix  and  $X_N(z)=X_N-zI_N$.
In this subsection, like in Theorem \ref{thres}, neither $n$ nor $C$ depend on $N$.

\begin{theorem}\label{matrixmain}
Let $n$ be fixed and let
$$
C\in\Comp^{n\times n},\quad A_N:= P_NCQ_N\in\Comp^{N\times N},
$$
be deterministic  matrices, where  $P_N\in\Comp^{N\times n}$, $Q_N\in\Comp^{n\times N}$, $N=1,2,\dots$
Let $X_N\in\Comp^{N\times N}$ be a random matrix. We assume that

 \begin{itemize}
 \item[(a1.2)] $X_N$ is either a  Wigner matrix from Example \ref{Wigner} or a random sample covariance matrix from Example \ref{MP}, so that the resolvent of $X_N-zI_N$ has a limit law $m_{\x}(z) I_N$ on the family of sets $\mathbf S^{\x}_{N,\omega}$ with the rate $\Psi^{\x}_N(z)$, where $\x\in\set{\W,\MP}$, respectively, and let $\mathbf{T}$ be a compact set that does not intersect the real line.
\item[(a2.2)] $C$ is invertible,
\item [(a4)]  $\norm{P_N}_2,\norm{Q_N}_2,c(P_N),r(Q_N)\leq \const$,
\item[(a5.2)] the matrix $D:=CQ_NP_N$
 is independent from $N$.
 \end{itemize}

Then the eigenvalues  of $X_N+A_N$ are with high probability outside the set
\begin{equation}\label{Smatrix}
\widetilde{\mathbf  S}^{\x}_{N,\omega}:=\set{ z\in\mathbf  S^{\x}_{N,\omega} :  \min_{\xi\in\sigma(D)} | 1+ \xi m_{\x}(z)  |^{p_\xi}\geq  N^{-\tau\omega}  }
\end{equation}
and
\begin{equation}\label{Tmatrix}
\widetilde{\mathbf{T}}_N:=  \set{z\in\mathbf  T :  \min_{\xi\in\sigma(D)} | 1+ \xi m_{\x}(z)  |^{p_\xi} \geq  N^{-\rho}}  ,\end{equation}
  where $\rho<\tau<\frac12$, $p_\xi$ denotes the size of the largest block corresponding to $\xi$ and $\sigma(D)$ is the set of eigenvalues of $D$.
The resolvent of the polynomial $A_N+X_N-zI_N$ has on $\widetilde{\mathbf{ S}}^{\x}_{N,\omega}$ and $\widetilde{\mathbf T}_N$ the following  limit law
$$
\widetilde{M}_N(z)=\big [ m_{\x}(z)I_N-m_{\x}^2(z)P_N(C^{-1}+m_{\x}(z)Q_NP_N)^{-1}Q_N\big],
$$
with the rates
$$
N^\frac\omega2 \Psi^{\x}_N(z)\quad \text{ and } \quad  N^\tau \Psi^{\x}_N(z),
$$
respectively.

Furthermore, if $z_0\in\Comp\setminus\Real$ is such that $\xi=-\frac 1{m_{\x}(z_0)}$ is an eigenvalue of
 $D$ with algebraic multiplicity $k_\xi$ and the size of the largest Jordan block is equal to $p_\xi$, then the $k_\xi$ eigenvalues $\lambda_{1}^N,\dots,\lambda^N_{k_\xi}$ of $X_N+A_N$ closest to $z_0$ are simple and converge to $z_0$ in the following sense
\begin{equation}\label{matrixconv}
|\lambda_{l}^N-z_0|\prec N^{-\frac1{2 p_\xi}},\quad l=1,\dots, k_\xi,
\end{equation}
  provided that the independent random variables constituting the matrix $X_N$, i.e. $W_{ij}$ in Example \ref{Wigner} or, respectively, $Y_{ij}$ in Example \ref{MP}, have continuous distributions.
\end{theorem}

\begin{proof}
\emph{ I: Limit law for the resolvent}

First we prove that the resolvent of $X_N+A_N$ contains with high probability the set $\mathbf{S}_{N,\omega}^{\x}$. We fix $\tau< \frac 12$ and let $\tau'\in(\tau,1/2)$. Let us note that the assumptions (a1)--(a4) of Theorem \ref{thres}  are satisfied with $\alpha=\omega/2$ and $\beta=\tau'\omega<\alpha$. With $K(z)$ as in Theorem \ref{thres} and $D=S J_D S^{-1}$, where $S$ is invertible and $J_D$ is in Jordan normal form, one has
\begin{eqnarray*}
\norm{K(z)^{-1}}_2&=&\norm{ (C^{-1} + m_{\x}(z)Q_NP_N )^{-1}}_2\\
&\leq&\norm{C}_2 \norm{ (I_n+ m_{\x}(z)D)^{-1}}_2\\
&\leq&\norm{C}_2 \norm{S}_2\norm{S^{-1}}_2 \norm{ (I_n+ m_{\x}(z)J_D)^{-1}}_2\\
&\leq & \norm{C}_2 \norm{S}_2\norm{S^{-1}}_2  n \maxnorm{ (I_n+ m_{\x}(z)J_D)^{-1}} .
\end{eqnarray*}
Note that $(I_n+ m_{\x}(z)J_D)^{-1}$ is a block-diagonal matrix with blocks corresponding to possibly different eigenvalues $\xi$ of $D$, of possibly different sizes $r$, of the form
$$
\frac 1{m_{\x}(z)} \matp{ s&1 &  &  \\  & \ddots & \ddots &  \\   &  & \ddots & 1 \\       & & & s }^{-1}
    = \frac 1{m_{\x}(z)} \matp{s^{-1} &-s^{-2}  & \dots & (-1)^{r+1}s^{-r} \\  & s^{-1}  & \dots & (-1)^{r}s^{-r+1} \\  &  & \ddots & \vdots \\  &  &  & s^{-1}}\in\Comp^{r\times r},
    $$
    where     $s=\frac{1+\xi m_{\x}(z)}{m_{\x}(z)}$ and the non-indicated entries are zeros.
Hence,
\begin{eqnarray*}
\maxnorm{ (I_n+ m_{\x}(z)J_D)^{-1}}
&\leq & \frac1{|m_{\x}(z)|} \max_{\xi\in\sigma(D)} \max_{j=1,\dots,p_\xi} \left| \frac{m_{\x}(z)}{1+\xi m_{\x}(z)}\right|^j\\
&\leq &  \max\set{1,|m_{\x}(z)|^{n-1}} \max_{\xi\in\sigma(D)} \max\set{1, \left| \frac{1}{1+\xi m_{\x}(z)}\right |^{p_\xi}}.\\
\end{eqnarray*}

As $m_{\x}(z)$ is bounded on $\bigcup_N\mathbf  S^{\x}_{N,\omega}$ we can apply estimates \eqref{ratesup1} and \eqref{ratesup1MP} and Theorem \ref{thres}  with $\alpha=\frac\omega2$  and get that the eigenvalues of $X_N+A_N$ are  with high probability outside the set
$$
\bigg\{z\in\mathbf{ S}^{\x}_{N,\omega}: {\norm{K_N(z)^{-1}}_2}<N^{\tau' \omega}\bigg\},
$$
which in turn contains the following sets
    \begin{eqnarray*}
        &&  \bigg\{z\in\mathbf{ S}^{\x}_{N,\omega}: \max_{\xi\in\sigma(D)} \max\set{1, \left| \frac{1}{1+\xi m_{\x}(z)}\right |^{p_\xi}} < \frac{N^{\tau'\omega}}  \delta  \bigg\}\\
           &\supseteq&  \bigg\{z\in\mathbf{ S}^{\x}_{N,\omega}: \max_{\xi\in\sigma(D)} \max\set{1, \left| \frac{1}{1+\xi m_{\x}(z)}\right |^{p_\xi}} < N^{\tau\omega}  \bigg\}\\
    &=&  \bigg\{z\in\mathbf{ S}^{\x}_{N,\omega}: \min_{\xi\in\sigma(D)} |1+\xi m_{\x}(z) |^{p_\xi} > N^{-\tau\omega}    \bigg\},
    \end{eqnarray*}
   where $\delta$ is an appropriate constant and  $N>N_0(\tau,\tau',\delta,\omega)$ is  sufficiently large.

  The formula for  the  limit law $\widetilde M_N(z)$  follows straightforwardly, let us discuss now the convergence rate. First consider the case
  $\mathbf{S}_N=\mathbf{S}^{\x}_{N,\omega}$ and apply  Theorem \ref{thres}.
  For this aim we need to check the additional assumption that $\widetilde{\Psi}_N(z)=N^{\frac\omega 2}\Psi_N^\star(z)$ converges to zero for each $z$.
  This is, however, satisfied due to $\Psi_N^\star(z)=\mathcal{O}(N^{-1/2})$ and $\omega<1$.

Now consider the case
 $\mathbf{S}_N=\mathbf{T}$, we apply Theorem \ref{thres} with $\alpha=\tau$ and $\beta=\rho$. Repeating the arguments from the previous case (with  \eqref{ratesup2}, \eqref{ratesup2MP} used instead of \eqref{ratesup1} and \eqref{ratesup1MP}) we see that  the eigenvalues of $X_N+A_N$ are with high probability outside $\widetilde{\mathbf T}_N$.
The formula for  the  limit law $\widetilde M_N(z)$  follows straightforwardly, to see the convergence rates observing that the additional assumption on convergence of $\widetilde\Psi_N(z)$ is satisfied due to $\tau<1/2$.\\

\emph{II: Eigenvalues (`Furthermore' part).}
First note that  due to the fact that $m_{\x}'(z_0)\neq 0$ the function
 $$
 \det K(z)= \det C^{-1} \det( I_n + m_{\x}(z) D  )
 $$
 has a zero of order $k_\xi$ at $z_0$. By Theorem \ref{spectrum}, $\det L_N(z)$ has exactly $k_\xi$ zeros, counting with multiplicities, $\lambda_1^N,\dots,\lambda_{k_\xi}^N$, that  converge to $z_0$ as
 \begin{equation}\label{matrixconv2}
|\lambda_{l}^N-z_0|\prec N^{-\alpha},\quad l=1,\dots, k_\xi
\end{equation}
 with some $\alpha>0$.
  Each of these zeros is an eigenvalue of $X_N+A_N$ converging to  $z_0$.
 We show now that $\alpha=\frac1{2p_\xi}$. Let us fix a compact set $\mathbf T$, not intersecting the real line, and such that $z_0$ is in the interior of $\mathbf T$. As $\frac{d(1+m_{\x}(z)\xi)}{dz} (z_0)\neq0$ one gets immediately from \eqref{matrixconv2} and the form of the set $\widetilde{\mathbf{  T}}_N$ in \eqref{Tmatrix} that $\alpha\geq \tau/p_\xi$ with arbitrary $\tau<\frac 12$.  Hence, \eqref{matrixconv2} holds with $\alpha=\frac1{2p_\xi}$ as well.

 Let us see that the zeros of $L(z)$ are almost surely simple. Note that
 $
 \det L(z)=\det C^{-1} \frac {\det(X_N-zI+A)}{\det(X_N-zI)}
 $
 is a rational function and if it has a double zero then  $\det(X_N-zI+A)$  has a double zero. However, it is a well-known fact that the eigenvalues of the random matrix with the continuously distributed entries are almost surely simple, see, e.g.,  \cite[Exercise 2.6.1]{taobook}.
Thus the  zeros $\lambda_1^N,\dots,\lambda_{k_\xi}^N$ are almost surely mutually different.

\end{proof}

\begin{remark}
The Theorem above may be easily generalised to the situation where the resolvent of $X_N-zI_N$ has a  limit law of the form $\mu(z)I_N$ with $\mu(z)$ being a Stieltjes transform of a probability measure. The only exception is the counting of the eigenvalues in \eqref{matrixconv}, namely it is not true that the eigenvalues converging to $z_0$ need to be simple and that their number has to be precisely $k_\xi$.
 The proof of this generalisation  follows exactly the same lines except the last two paragraphs. 
\end{remark}

\begin{example}\label{Wignerplusi}
Let $X_N$ be the Wigner matrix as in Example \ref{Wigner}. For the simulations Wigner matrices with real Gaussian entries were used. We use the notation from Theorem \ref{matrixmain}.
We compare the convergence rates of eigenvalues for the following four instances of the matrix $C$
$$
C^{(1)}=[8\ii],\quad C^{(2)}=\diag(8\ii,8\ii,8\ii),\quad C^{(3)}=\matp{8\ii & 1& 0\\ 0 & 8\ii & 1\\ 0 & 0 & 8\ii},\quad C^{(4)}=[2\ii],
$$
with
$$
Q_N^{(1)}=[I_n, 0_{n,N-n} ],\  P_N^{(1)}=Q_N^{(1)*},\quad A^{(j)}_N=P_N^{(1)}C^{(j)}Q_N^{(1)},\quad  j=1,2,3,4.
$$
 We have  $D^{(j)}=C^{(j)}$, $j=1,2,3,4$ and $\xi^{(j)}=8\ii$, $j=1,2,3$ and, $\xi^{(4)}=2\ii$. It is a matter of a straightforward calculation that the only solution of $1+\xi^{(j)} m_{\W}(z)=0$ equals $63\ii/8$ for $j=1,2,3$ and $3\ii/2$ for $j=4$. A sample set $\mathbf{\tilde{S}}_N$ and the spectrum of $X_N+A_N^{(4)}$ are plotted in Figure \ref{F0}.
According to Theorem \ref{matrixmain},  the rate of convergence is (simplifying the statement slightly) $N^{-\frac12}$ in the $C^{(1)}$, $C^{(2)}$, and $C^{(4)}$ case, and $N^{-\frac16}$ in the $C^{(3)}$ case.
The graphs of
$$
\delta(N):=\max_{j=1,\dots, k_\xi}  |\lambda^N_j -z_0|
$$
are presented in Figure \ref{F1}. Note that the log-log plot support the conjecture that the exponents in estimates of the convergence rates cannot be in practice improved. Namely, in all four cases the slope of the corresponding group of points are approximately ${-1/2}$ in cases  (1), (2) and (4) and $-1/6 $  case (3).
Furthermore, it is visible that  while the exponent in the rate of convergence in the cases (1), (2) and (4) is the same, $ N^{-\frac12}$, these rates may differ by a constant. This is visible as a vertical shift of the graphs in corresponding to cases (1), (2) and (4).

\begin{figure}
\includegraphics[width=250pt]{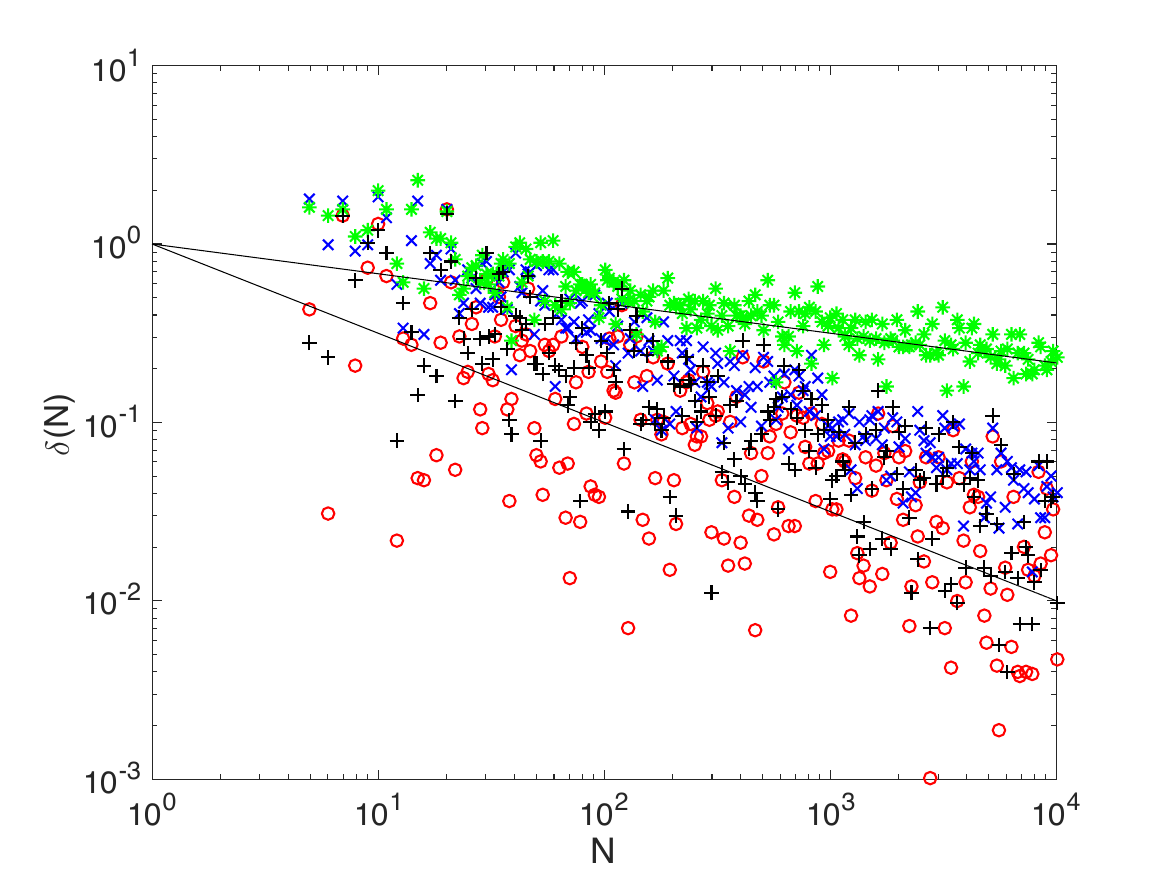}
\caption{The log-log plots for $C^{(1)}$  (red circle), $C^{(2)}$ (blue cross), $C^{(3)}$ (green star) and $C^{(4)}$ (black plus),  see Example \ref{Wignerplusi}. The plots of $N^{-1/2}$ and $N^{-1/6}$  are marked with black lines for reference. }\label{F1}
\end{figure}
\end{example}

\begin{figure}
\includegraphics[width=250pt]{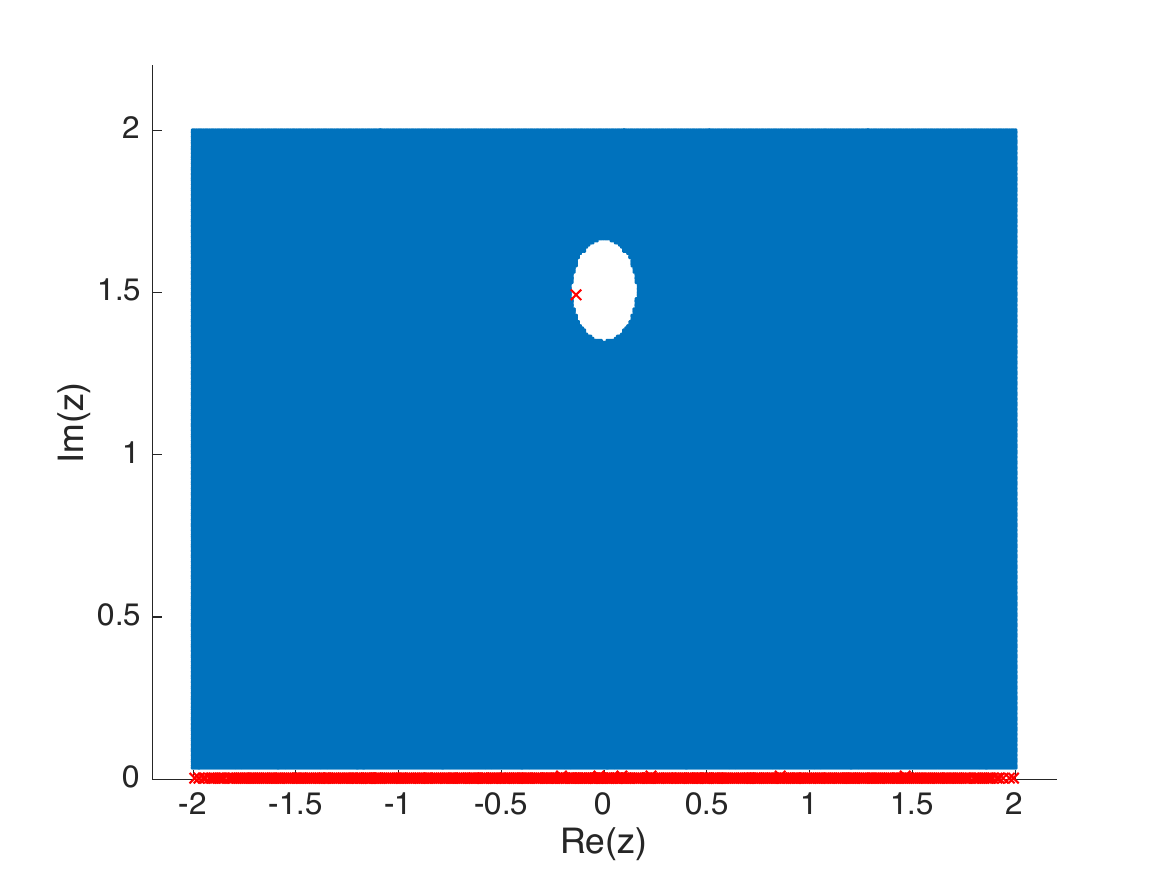}
\caption{The set  $ \widetilde{\mathbf  S}^{W}_{1000}$  (in blue) and the spectrum of $W_{1000}+A_{1000}^{(4)}$ (red crosses)  from Example \ref{Wignerplusi}.  }\label{F0}
\end{figure}

\begin{figure}
\includegraphics[width=250pt]{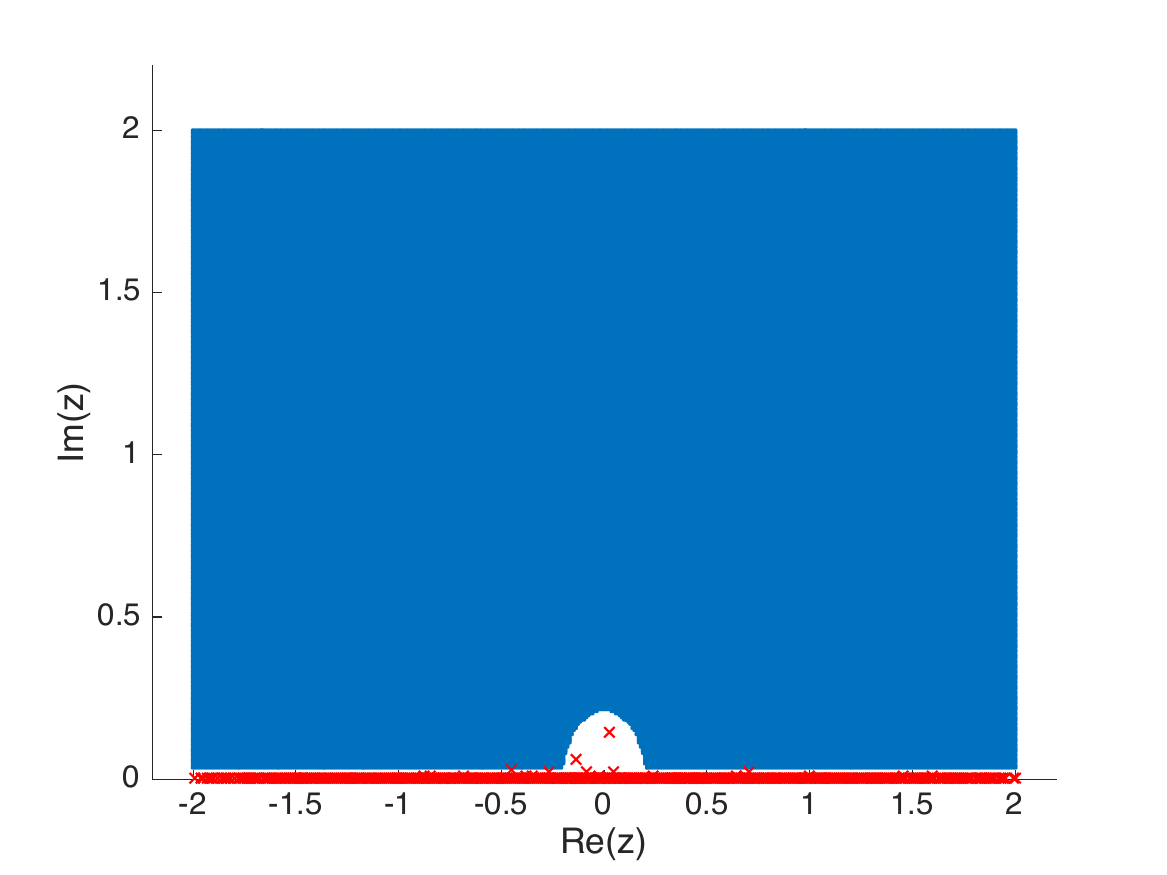}
\caption{The set  $\tilde{\mathbf{S}}^{W}_{1000,\omega}$ (in blue) and the spectrum of $W_{1000}+A_{1000}^{(5)}$ (red crosses)  from Example \ref{Wignerplusi2}.  }\label{F0b}
\end{figure}

Let us formulate now a direct corollary from Theorem \ref{matrixmain}.

\begin{corollary}\label{QP=0}
If, additionally to the assumptions of Theorem \ref{matrixmain}, $Q_NP_N=0$,
then the eigenvalues  of $X_N+A_N$ are with high probability outside the set $\mathbf{S}^{\x}_{N,\omega}$  and the resolvent of $X_N+A_N$ has on $\mathbf{S}^{\x}_{N,\omega}$ the following limit law
$$
\widetilde{M}_N(z)=\big [ m_{\x}(z)I_n-m_{\x}^2(z)A_N \big],
$$
with the rate
$$
\widetilde{\Psi}^{\x}_N(z):= N^\frac12 \Psi^{\x}_N(z).
$$
\end{corollary}

\begin{example}\label{Wignerplusi2}
In this example we will compare the convergence of the eigenvalues of $X_N+A_N$ to the real axis in three different situations. Here $X_N$ is again a Wigner matrix. First let us take the matrix  $A^{(1)}_N=P_N^{(1)} C^{(1)}Q_N^{(1)}$  from Example \ref{Wignerplusi}.
In this situation there is one  eigenvalue $\lambda_1^N$ of $X_N+A_N^{(1)}$ converging to $z_0=63\ii/8$, cf. Example \ref{Wignerplusi}, and the set
 $\widetilde{\mathbf S}_{N,\omega}^{\W}$ is for large $N$ a rectangle with an (approximately) small disc around  $z_0$ removed, similarly as for $C^{(4)}$ in Figure \ref{F0}.
  The other eigenvalues converge to the real line with the rate $N^{-1}$, i.e. $\Delta(N)\prec N^{-1}$ where
 \begin{equation}\label{Delta1}
\Delta(N)=\max\set{ | \IM \lambda| : \lambda\in\sigma(X_N+A_N)\setminus \set{ \lambda_{1}^N}}.
\end{equation}
By definition of $\mathbf{S}_{N,\omega}^{\W}$ one can see that $\Delta(N)\prec N^{-1+\omega}$ with the arbitrary parameter $\omega\in(0,1)$. However, by the definition of stochastic domination, it means that $\Delta(N)\prec N^{-1}$, see Remark \ref{whyN^-1}.
The numbers $\Delta(N)$ are plotted in Figure \ref{F2}. One can observe that the plot bends in the direction of the the line given by $N^{-1}$, which is still  in accordance with the definition of stochastic domination.

The second situation to consider is $A^{(5)}_N=P_N^{(1)} C^{(5)}Q_N^{(1)}$ with
$$
C^{(5)}= [i], \quad Q^{(1)}_N=[I_1, 0_{1,N-1} ], P^{(1)}_N=Q_N^*
$$
In this situation the equation $1+\xi m_{\W}(x)$, with $\xi=\ii$, has no solutions in $\Comp\setminus[-2,2]$. However,  $z_0=0$ can be seen as a solution,  if we define $m_{\W}(0)$ as $\lim_{y\downarrow 0} m_{\W}(y\ii )=1i$. Hence,  the set   $\widetilde{\mathbf S}^{\W}_N$ is a rectangle with an (approximately) half-disc around $z_0=0$ removed, see Figure \ref{F0b}. The half disc has radius of order $N^{-\frac12}$, hence
 \begin{equation}\label{Delta2}
\Delta(N)=\max\set{ | \IM \lambda| : \lambda\in\sigma(X_N+A_N)}
\end{equation}
converges to zero with the rate $N^{-\frac12}$, which can be seen in Figure \ref{F2}.

The last situation to consider is $A^{(6)}_N=P_N^{(2)} C^{(5)}Q_N^{(2)}$ with
$$
C^{(5)}= [i], \quad Q^{(2)}_N=[0,1, 0_{1,N-2} ], P^{(2)}_N=[1,0,0_{1,N-2}]^\top.
$$
Here, according to Corollary \ref{QP=0} the sets $\mathbf{S}^{\W}_N$ and $\widetilde{\mathbf S}^{\W}_N$ coincide.  Hence, all the eigenvalues converge to the real axis with the rate $N^{-1}$, the same comments concerning $\omega$ as in the $C^{(4)}$ case above apply. The plot of $\Delta(N)$, defined as in \eqref{Delta2}, can be seen in Figure \ref{F2}.

\begin{figure}
\includegraphics[width=250pt]{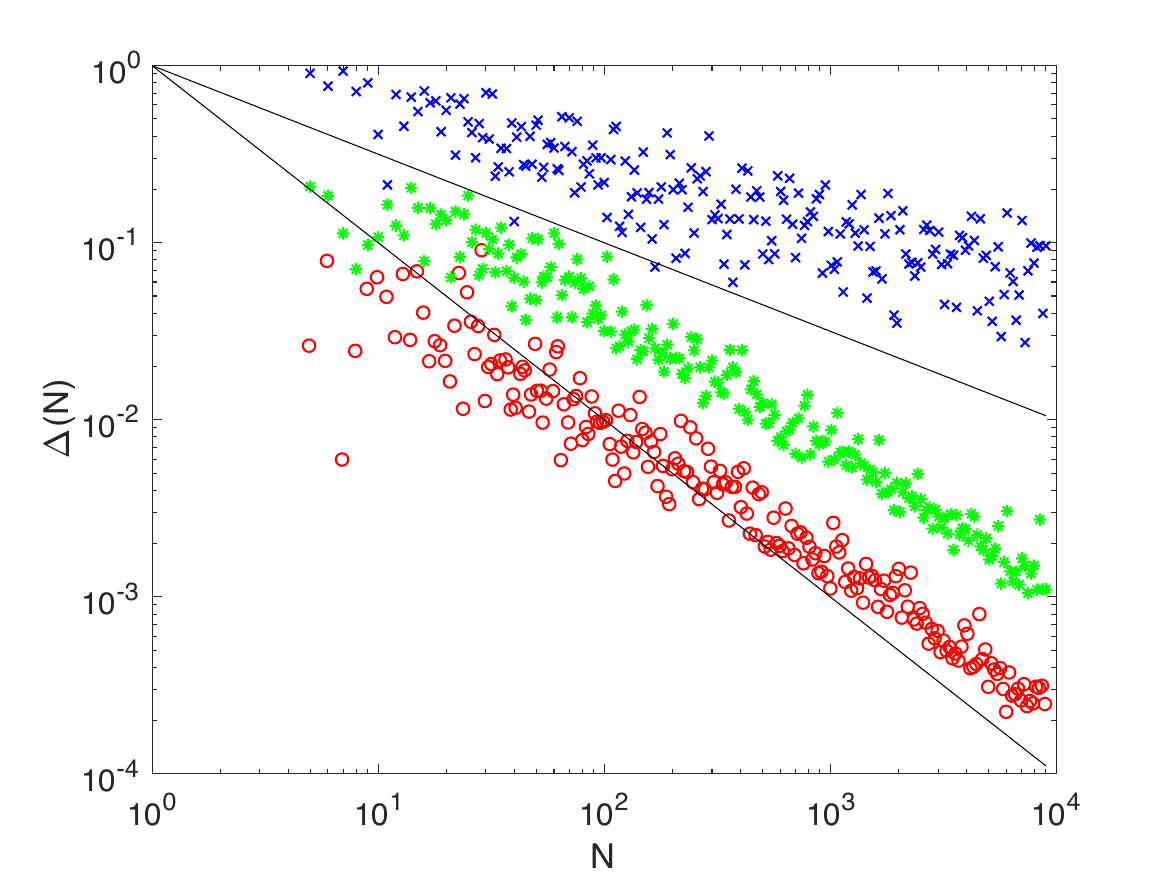}

 \caption{ The log-log plots of $\Delta(N)$ for $A_N^{(1)}$   (red circle), $A_N^{(5)}$ (blue cross) and $A_N^{(6)}$ (green star),  see Example \ref{Wignerplusi2}. The plots of $N^{-1}$  and $N^{-1/2}$  are marked with black lines for reference. }\label{F2}
\end{figure}
As in Example \ref{Wignerplusi}, the plot suggest that the exponent in the convergence rate  cannot be improved in the discussed examples.

\end{example}

The next corollary will concern the class of port Hamiltonian matrices, i.e. matrices of the form $A-Z$, where $A=-A^*$ and $Z$ is positive definite. This class has recently gathered some interest \cite{mehl2016stability,MehMS16b} due to its role in mathematical modelling. Clearly, the spectrum of $A-Z$ lies  in the closed left half-plane. We will consider below the case where $A=C\oplus 0_{N-n,N-n}$ is a nonrandom matrix with $n$  fixed
and $Z=Y^*Y$ is the random sample covariance matrix.
For the sake of simplicity we will take a square random sample covariance matrix ($N=M$).

\begin{corollary}
Let $Z_N=Y_N^*Y_N\in \Comp^{N\times N}$ be a random sample covariance matrix from Example \ref{MP} with $M=N$. Let $n>0$ be fixed and let $A_N=C\oplus 0_{N-n,N-n}$, where $C\in \Comp^{n\times n}$ is a skew-symmetric matrix $C=-C^*$ with nonzero eigenvalues $\ii t_j$ with algebraic multiplicities, respectively, $k_j$ $(j=1,\dots r,\ k_1+\dots +k_r=n)$.
Let
$$
z_{j} := - \frac{t_j^2}{1+\ii t_j},\quad j=1,\dots r.
$$
Then, for any $j=1,2,\dots,r$, the  $k_j$ eigenvalues $\lambda_{j,1},\lambda_{j,2},\dots,\lambda_{j,k_j}$ of $A_N-Z_N$ converge in probability to $z_j$  as
$$
|\lambda_{j,l}^N-z_j|\prec N^{-\frac12},
$$
where $l\in\{1,2,\dots,k_j\}$.
\end{corollary}

\begin{proof}
Consider the matrix $Z_N-A_N$. As $\ii C$ is Hermitian, the matrix $D=-C$ from Theorem \ref{matrixmain} does not have any Jordan chains longer than one.
Note that  $-z_j$ is  the solution of $1+\ii t_jm_1(z)=0$.
The claim follows now directly from Theorem \ref{matrixmain}.
\end{proof}

\section{Random matrix pencils and $H$-selfadjoint random matrices}\label{HW}

In this section we will employ the setting of random matrix pencils. 
The theory is aimed on localisation of the spectrum of the products of matrices $H_NX_N$.
Although the linear pencil  appear only in the proof of the main result (Theorem \ref{MainForX_N2}), its role here is crucial.
 In what follows $H_N$ is a nonrandom diagonal matrix
$$
H_N=\diag(c_1\dts c_{n_N})\oplus I_{N-{n_N}},\quad c_1\dts c_{n_N}<0,
$$
and $X_N$ is a generalised Wigner or random sample covariance matrix.
To prove a limit law for the resolvent we need $n_N$ to be a slowly increasing sequence, while to count the number of eigenvalues converging to their limits we need $n_N$ to be constant. Note that unlike in the case of perturbations $X_N+A_N$ considered in Theorem \ref{matrixmain} we do not need to localise the spectrum near the real line, as  the spectrum of $X_N$ is symmetric with respect to the real line and contains at most $n_N$ points in the upper half plane, see e.g. \cite{GLR}. The following theorem explains the behavior of all non-real eigenvalues of $H_NX_N$. It covers the results on locating the nonreal eigenvalues of $H_NX_N$ of \cite{MWPP} and \cite{Wojtylak12b}, where the case $n_N=1$ was considered. In addition, the convergence rate and formula for the resolvent are obtained.

\begin{theorem}\label{MainForX_N2}
Let $n_N \leq \log N$ be a sequence of nonrandom natural numbers and let
$$
H_N=\diag(c_1\dts c_{n_N})\oplus I_{N-{n_N}},
$$
where $(c_j)_j$ is a negative sequence such that the sequence $(c_j^{-1})_j$ is bounded.
 We also assume that
  \begin{itemize}
 \item[(a1.3)] $X_N$ is either a Wigner matrix from Example \ref{Wigner} or a random sample covariance matrix from Example \ref{MP}, so that the resolvent of $X_N-zI_N$ has a limit law $m_{\x}(z) I_N$ on the compact set $\mathbf{T}$ with the rate $\Psi^{\x}_N(z)$, where $\x\in\set{\W,\MP}$, respectively.
\end{itemize}

Then, with high probability, the eigenvalues of the matrix $H_NX_N$ are outside the set
$$
\widetilde{\mathbf{T}}_N:=\set{z\in\mathbf{T}: \left| \frac{c_j}{(c_j-1)z} +m_{\x}(z) \right|^{-1}<N^{\beta}, \text{ for } j=1,2,\dots,n_N },
$$
where $\beta<\alpha<\frac12$.
Furthermore, the resolvent of $H_NX_N-zI_N$ has a  limit law on $\widetilde{\mathbf{T}}_N$
\begin{equation}\label{resM}
\widetilde{M}_N(z)=\diag(g_1(z),\dots, g_{n_N}(z))\oplus m_{\star}(z)I_{N-n_N}, \quad z\in\widetilde{\mathbf{T}}_N,
\end{equation}
with
$$
g_j(z)=m_{\star}(z)\frac{1}{(c_j-1)zm_{\star}(z)+c_j}
$$
and the rate $n_N N^{\alpha}\Psi_N^{\star}(z)$.

If, additionally, $n_N=n$ is constant and if $k_j$ denotes the number of repetitions of $c_j$ in the sequence $c_1,\dots,c_n$ and if, using the notation of Example \ref{MP},
$$
z_{j} :=
 \begin{cases} \frac{-c_j}{\sqrt{1-c_j}}\ii & : \x=\W   \\
\frac{-(c_j-1)^2\gamma_-\gamma_++(2c_j\phi^{-\frac12}+(c_j-1)(\phi^{\frac12}-\phi^{-\frac12}))^2}{2(2c_j\phi^{-\frac12}+(c_j-1)(\phi^{\frac12}-\phi^{-\frac12}))(c_j-1)-(c_j-1)^2(\gamma_++\gamma_-)}
& : \x=\MP\end{cases},\quad  j=1,\dots,n
$$
then, for any $j=1,2,\dots,n$, there are exactly  $k_j$ eigenvalues $
\lambda_{1}^N,\dots,\lambda_{k_j}^N$ of $H_NX_N$
converging to $z_j$ in probability as
$$
|\lambda_{l}^N-z_j|\prec N^{-\frac{1}{2}},\quad l=1,\dots k_j.
$$
\end{theorem}

\begin{proof}
 First note that for $z\in\mathbf{T}$ we have
\begin{equation}\label{H3}
H_NX_N-zI_N=H_N(X_N-zI_N+zP_NC_NQ_N),
\end{equation}
where
$$
C_N=\diag(1-c^{-1}_1,\dots,1-c^{-1}_{n_N}),\quad P_N=\matp{I_{n_N}\\ 0_{N-n_N,n_N}},\quad  Q_N=P_N^*.
$$
Note that, by \eqref{ratesup2} and \eqref{ratesup2MP}, the polynomials $X_N(z)=X_N-zI_N$, $A_N(z)=zP_NC_NQ_N$ satisfy the assumptions of Theorem \ref{thres} with any
$\alpha<\frac12$.

Hence, the resolvent of $H_N(X_N-zI_N+zP_NC_NQ_N)$ has a limit law
\begin{eqnarray*}
\widetilde{M}(z)&=& \Big(m_{\x}(z)I_N - m_{\x}(z)P_N\Big( (zC_N)^{-1}+Q_Nm_{\x}(z)I_NP_N \Big)^{-1}Q_Nm_{\x}(z)\Big) H_N^{-1}\\
&=& m_{\x}(z)\diag(c^{-1}_1\dts c^{-1}_{n_N})\oplus I_{N-{n_N}} \\
- m_{\x}(z)^2\!\!\!\!\!& \cdot \null&\!\!\!\!\!\left(\Big( \diag\big(\frac{1}{z(1-c^{-1}_1)},\dots,\frac{1}{z(1-c^{-1}_{n_N})}\big)+m_{\x}(z)I_{n_N} \Big)^{-1}\oplus 0_{N-n_N}\right)H^{-1}_N\\
&=& \diag(g_1(z),\dots, g_{n_N}(z))\oplus m_{\x}(z)I_{N-n_N}
\end{eqnarray*}
on the set $\widetilde{\mathbf T}_N$ with $\beta<\alpha<\frac12$. The rate of this convergence is $n_N N^{\alpha}\Psi_N^{\x}(z)$,
due to the fact that $H_N^{-1}$ is a diagonal matrix with bounded entries.
Finally, note that $n_N N^{\alpha}\Psi^x_N(z)$ converges to zero pointwise in $z$.

Now let us prove the statement concerning the eigenvalues. By  \eqref{H3} the eigenvalues of $H_NX_N$ are the eigenvalues of the linear pencil $X_N-zI_N+zP_NC_NQ_N$. Define $K(z)$  as in Theorems \ref{thres} and \ref{spectrum}:
$$
K(z):=(Cz)^{-1}+m_{\x}(z) I_n.
$$
As the matrix $C$ is diagonal we have that $\det K(z)=0$ is equivalent to
 \begin{equation}\label{ewig}
 \frac{c_j}{c_j-1}\frac{1}{z}+m_{\x}(z)=0,
 \end{equation}
 for some $ j\in\set{1,\dots n}$. Consequently, the points $z_j$, $j=1,\dots, n$  are precisely the zeros of $\det K(z)$.

If $X_N$ is a  Wigner matrix, then using the well known equality
 \begin{equation} \label{m+1/m+z=0}
m_{\W}(z)+\frac{1}{m_{\W}(z)}+z=0
\end{equation}
it is easy to show that the equation \eqref{ewig} has for each $j=1,\dots n$
two complex solutions
$$
z_j^\pm = \pm \frac{c_j}{\sqrt{1-c_j}}i,
 $$
 let $z_j=z_j^-$. If $X_N$ is a random sample covariance matrix then direct computations give the formula for $z_j$. By Theorem \ref{spectrum}, for each $z_j$, there are $k_j$ eigenvalues of $X_N$ converging to $z_j$ as
 \begin{equation}\label{est1}
|\lambda_{l}^N-z_{j}|\prec N^{-\alpha},\quad l=1,\dots k_j,
\end{equation}
 with some $\alpha>0$.
By the first part of the theorem, with high probability, the eigenvalues of $H_NX_N$ are outside  the set
 \begin{eqnarray*}
 \widetilde{\mathbf{T}}_N&=&\set{z\in\mathbf{T}: \left| \frac{c_j}{(c_j-1)z} +m_{\x}(z) \right|^{-1}<N^{\beta}, \text{ for } j=1,2,\dots,n}\\
 &\supseteq&\set{z\in\mathbf{T}: | z-z_j| > N^{-\beta'}, \text{ for } j=1,2,\dots,n}
 \end{eqnarray*}
where $\beta'<\beta<\frac12$. Hence, $\alpha$ in \eqref{est1} can be chosen as arbitrary $\beta'<\frac12$. Hence by the definition of stochastic domination (see Definition \ref{stochdom}), equation \eqref{est1} holds with $\alpha=\frac12$ as well.

\end{proof}

\begin{remark}
  Let us note two facts about the formula for $z_j$ in Theorem \ref{MainForX_N2} above.

  In the Wigner matrix case the point $z_j$ is in the upper half-plane and its complex conjugate is also a limit point of $k_j$ eigenvalues of $H_NX_N$ due to the symmetry of spectrum of $H_NX_N$ with respect to the real line.

  Further, in the random sample covariance matrix case it holds that $z_j<0$. Further, if the underlying matrix is a square matrix, then the formula simplifies to
  $$
  z_j=\frac{c_j^2}{c_j-1}.
  $$
  \end{remark}

\begin{example}
Let us consider $H=\diag(-1,-2,-2,1,1,\dots,1)\in \Comp^{N\times N}$.
The spectrum of the matrix $X_N=H_NW_N$ is symmetric to the real line and there are three pairs of eigenvalues which do not lie on the real line. The rest of the spectrum is real.
By Theorem \ref{MainForX_N2} the resolvent of $X_N$ converges in probability to
$$
\widetilde{M}(z)=\diag\Big(\frac{-m_{\W}(z)}{2zm_{\W}(z)+1},\frac{-m_{\W}(z)}{3zm_{\W}(z)+2},\frac{-m_{\W}(z)}{3zm_{\W}(z)+2},1,\dots,1 \Big).
$$
 Moreover, one eigenvalue of $X_N$  converges in probability to $z_1=\frac{\sqrt{2}}{2}\ii$ and two eigenvalues converge in probability to $z_2=\frac{2\sqrt{3}}{3}\ii$.

\end{example}

We conclude the section with a different type of a random pencil.

\begin{remark}\label{eUU}
We recall the method of detecting damped oscillations in a noisy signal via Pad\'e approximations of the Z-transform of the signal, proposed in \cite{bessis2009universal}. The method has found several practical applications \cite{practical1,practical2,practical3,practical4}, its numerical analysis can be found in \cite{bessis2009universal}. Here finding the  limit law (if it exists) for the resolvent of  the pencil $zU_0-U_1$, where
$$
  U_0=\matp{  s_0 &   s_1 & \dots &   s_{n-1}\\
      s_1 && \iddots &   s_{n}\\ \vdots&\iddots&\iddots&\vdots \\   s_{n-1} &   s_{n}&\dots  &   s_{2n-2}},\quad U_1=\matp{  s_1 &   s_2 & \dots &   s_{n}\\
      s_1 && \iddots &   s_{n}\\ \vdots&\iddots&\iddots&\vdots \\   s_n &   s_{n+1}&\dots  &   s_{2n-1}},
$$
and  $s=s_0,\dots s_{2n-1}$ is a white noise, would substantially contribute to the signal analysis, via Theorem \ref{thres} above, see \cite{Barone2005,Barone2012} for details. The spectral properties of Hankel matrices were studied e.g. in \cite{Byrc}. The investigation of the pencil $zU_0-U_1$ is left for subsequent papers.
\end{remark}

\section{Analysis of some random quadratic matrix polynomials}\label{spoly}

In the present section we will consider the spectra of matrix polynomials of the form
\begin{equation}\label{polyform}
X_N- p(z)I_N + q(z)u_N u_N^*,
\end{equation}
where $ X_N$ is either a Wigner or a random sample covariance matrix and $u_N$ is some deterministic vector and $p(z)$ and $q(z)$ are some (scalar-valued) polynomials and polynomials of type
 $$
  z^2(C_n\oplus0_{N-n}+I_N)+zX_N+D_n\oplus0_{N-n},
$$
where $C_n,D_n\in\Comp^{n\times n}$ are diagonal deterministic matrices.
This we see as a step forward in a systematic study of polynomials with random entries.
 The problem  of localising the spectrum of \eqref{polyform} is a clear extension of the usual perturbation problem for random matrices, see e.g. \cite{BGRnew,BCnew,cap1,FP,Knew,MMRR1,MMRR2,MMW2,Pizzo,Tnew}. Indeed, the latter problem can be seen as studying the pencil
$X_N- zI_N + u_N u_N^*$. In Example \ref{polyex} below we will demonstrate an essential difference between these two problems. The general case, i.e. following the setting of Theorem \ref{thres}, is as yet unclear and requires developing new methods for estimating the expression $\norm{K(z)^{-1}}_2$ appearing therein.

Note that matrix polynomials of  type $A- p(z)I_N + q(z)u u^*$ ($A\in\Comp^{n\times n}$, $u\in\Comp^n$) appear in many practical problems connected with modelling, cf. \cite{BHMST}.

\begin{theorem}\label{coruuu}
Let
$$
C(z)=q(z)  \in\Comp^{1\times 1}[z],\quad A_N(z):= q(z)u_Nu_N^* \in\Comp^{N\times N}[z],
$$
be deterministic  matrix polynomials, where  $u_N\in\Comp^{N}$, $N=1,2,\dots$.
Let  $X_N(z)=X_N- p(z) I_N \in\Comp^{N\times N}[z]$ be a random matrix polynomial, and let neither $p(z)$ nor $q(z)$ depend on $N$. We assume that

 \begin{itemize}
 \item[(a1.2)] $X_N$ is either a  Wigner matrix from Example \ref{Wigner} or a random sample covariance matrix from Example \ref{MP}, so that the resolvent of $X_N$ has a local limit law $m_{\x}(z) I_N$ on the family of sets $\mathbf S^{\x}_{N,\omega}$ with the rate $\Psi^{\x}_N(z)$, where $\x\in\set{\W,\MP}$, respectively, and let $\mathbf{T}$ be a compact set that does not intersect the real line.
\item[(a2.3)] $p(z)$ and $q(z)$ are fixed nonzero polynomials,
\item [(a4.3)] $u_N$ is a deterministic vector of norm one, having at most $n$ nonzero entries, where $n$ is fixed and independent from $N$.
 \end{itemize}

Then the eigenvalues of $X_N(z)+A_N(z)$ are with high probability outside the set
$$
\widetilde{\mathbf{S}}_{N,\omega}^{\x}=\set{z\in\Comp: p(z)\in\mathbf{S}_{N,\omega}^{\x},\ | m_{\x}(p(z))+q(z)^{-1}| >N^{-\beta\omega}}.
$$
and
\begin{equation}\label{Tmatrix1}
\widetilde{\mathbf{T}}_N:=  \set{z\in\Comp: p(z)\in \mathbf  T : | m_{\x}(p(z))+q(z)^{-1}| > N^{-\beta} } ,\end{equation}
 where  $\beta<\alpha<\frac12$.
The resolvent of the polynomial $X_N(z)+A_N(z)$ has on $\widetilde{\mathbf{ S}}^{\x}_{N,\omega}$ and $\widetilde{\mathbf T}_N$ the following  limit law
$$
\widetilde{M}_N(z)=m_{\x}(p(z))I_N -\frac{m_{\x}(p(z))^2}{m_{\x}(p(z))+q(z)^{-1}}u_Nu_N^*,
$$
with the rates
$$
N^\frac\omega2 \Psi^{\x}_N(z)\quad \text{ and } \quad   N^\alpha \Psi^{\x}_N(z),
$$
respectively.

Furthermore, for each solution $z_0$ with $p(z_0)\in\Comp\setminus\Real$ of the equation $m_{\x}(p(z))+q(z)^{-1}=0$ there exists  eigenvalues $\lambda^N_j$, $j=1,\dots,k$, where $k\geq 1$, of $X_N(z)$ converging to $z_0$ as
$$
|z_0-\lambda_j^N|\prec N^{-1/2},\quad j=1,\dots, k.
$$
\end{theorem}

\begin{proof} For the proof we note that the assumptions of Theorem \ref{thres} are satisfied, $X_N(z)$ has a limit law $m_{\x}(p(z))I_N$
and that
$$
K(z)= m_{\x}(p(z))+q(z)^{-1},
$$
and so the first part of the claim follows directly.

The statement concerning the convergence of eigenvalues follows from Theorem \ref{spectrum} and from the form of the set $\widetilde{\mathbf{T}}_N$, cf. the proofs of Theorems \ref{matrixmain} and \ref{MainForX_N2}.

\end{proof}

\begin{example}\label{polyex}
We present an example showing how the techniques above may be useful in localising eigenvalues of second order matrix polynomials.
Consider a particular matrix polynomial \eqref{polyform}:
\begin{equation}\label{Polyform}
P(z)=-4\pi^2\left( I_N - \frac12 e_Ne_N^\top\right)z^2 + {2\pi }\cdot e_Ne_N^\top z + 2I_N-e_Ne_N^\top - W_N,
\end{equation}
where $W_N$ is the real  Wigner matrix. 
Note that a similar matrix can be found in \cite{BHMST} as Problem \verb-acoustic_wave_1d- (after substitution $z=\frac{\lambda}N$), but with the Wigner matrix replaced by
$$
 \matp{   &-1&& \\ -1 &&\ddots  & \\ &\ddots&&-1\\&&-1&}.
$$
The matrix polynomial  in \eqref{Polyform} is real and symmetric, hence its spectrum is symmetric with respect to the real line.
Let us note that the spectrum is  a priori not localised  on the real and imaginary axis, for small $N$ we can have eigenvalues with relatively large both real and imaginary parts. However, these eigenvalues will converge to the real and imaginary axis as $N$ grows.
To see this, first note that the sets  $\tilde{S}^W_{N,\omega} \subset p^{-1}(S^W_{N,\omega})$ are contained in the first and third quadrant of the plane. However, due to the symmetry of the spectrum, we may extend the set to
\begin{equation}\label{SSS}
\set{z\in\Comp: z\in \tilde{S}^W_{N,\omega} \text{ or }\bar z\in \tilde{S}^W_{N,\omega}},
\end{equation}
 which lies in all four quadrants of the complex plane. The spectrum of $P(\lambda)$ is located with high probability in the complement of  the  set \eqref{SSS}.  A sample set \eqref{SSS} is plotted in Figure \ref{F5}. Note that the complement of \eqref{SSS} contains both the real and the imaginary axis with some neighbourhood, which is not clearly seen in  the picture. Hence, for large $N$, the spectrum is either (approximately) on the real or imaginary axis. Furthermore, the real spectrum concentrates on  $
p^{-1}([-2,2])=\left[-\frac{1}{\pi},\frac{1}{\pi}\right]$ and there is also  one  real eigenvalue (near $z=0.5$), which converges to a real solution of
\begin{equation}\label{mpq}
m_{\W}(4\pi^2z^2-2)+\frac1{2\pi^2z^2+ {2\pi\ii}\ z -1}=0
\end{equation}
 (we extend the function $m_{\W}(z)$ onto the real line as $m_{\W}(x)=\lim_{y\downarrow 0}m(x+\ii y)$, similarly for the matrix $A^{(5)}$ in Example \ref{Wignerplusi2}).

Due to its role in applications, computation of spectra of quadratic polynomials  is currently an important task, see e.g.  \cite{BHMST,Mac22006,MT,acoustic,Tisseur}.
The usual procedure is a linearisation (see \cite{Mac2006,linearization}), which,  however,  has some limitations  \cite{higam}.
Above we obtained a family of matrix polynomials (i.e. one coefficient is a random matrix) for which  we are able to  control the real part of the non-real eigenvalues. This  property can be useful e.g. in testing particular numerical  algorithms, by plotting maximum of $|\RE \lambda|$, over all non-real eigenvalues $\lambda$. If the algorithm works properly then this quantity should converge to zero as approximately $N^{-1}$.
A preliminary picture in \texttt{matlab} (Figure \ref{F6}) of maximum of $|\RE \lambda|$, over all non-real eigenvalues $\lambda$, does not reveal any numerical anomalies for $N\leq 10^3$ and the upper bound of order $N^{-1}$ is visible.

\begin{figure}
\includegraphics[width=250pt]{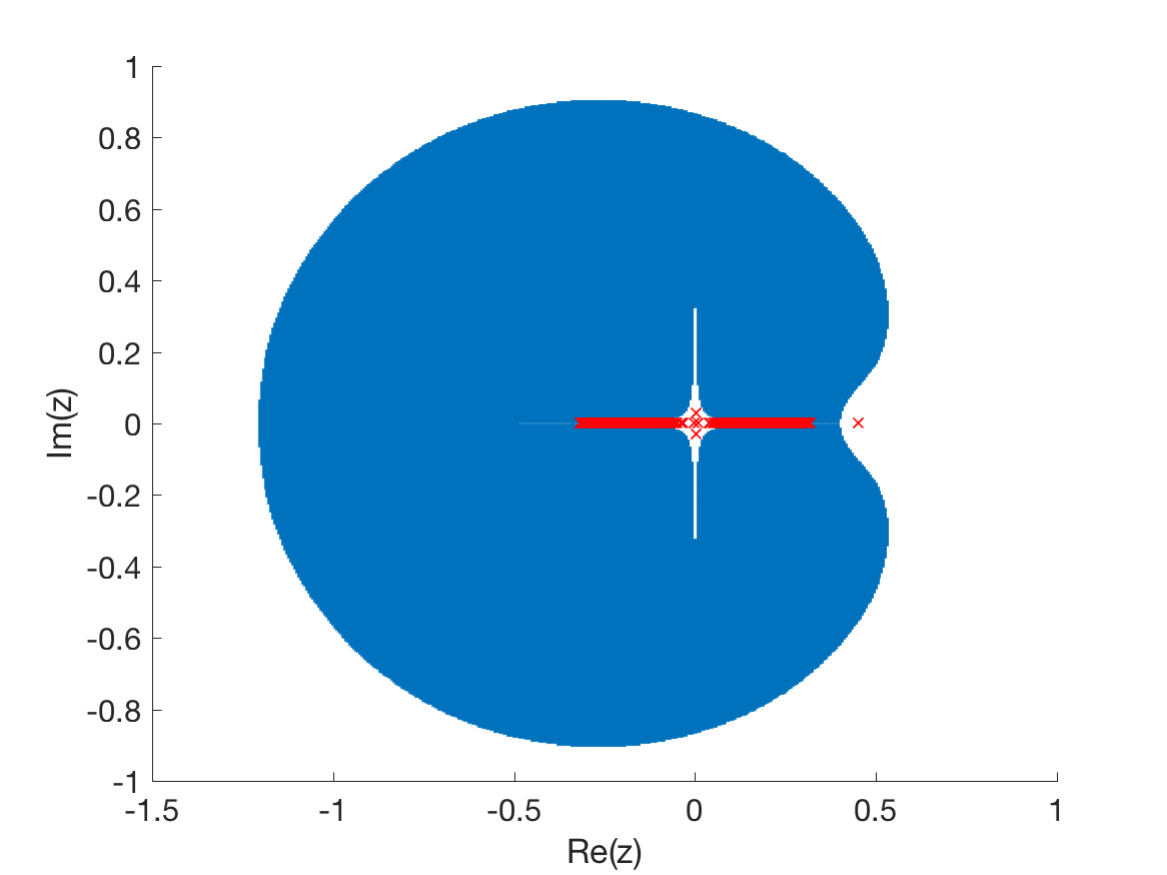}
\caption{ The spectrum of the quadratic random polynomial  \eqref{Polyform} with $N=500$ (red crosses) and the set \eqref{SSS} (blue).}\label{F5}
\end{figure}

\begin{figure}
\includegraphics[width=250pt]{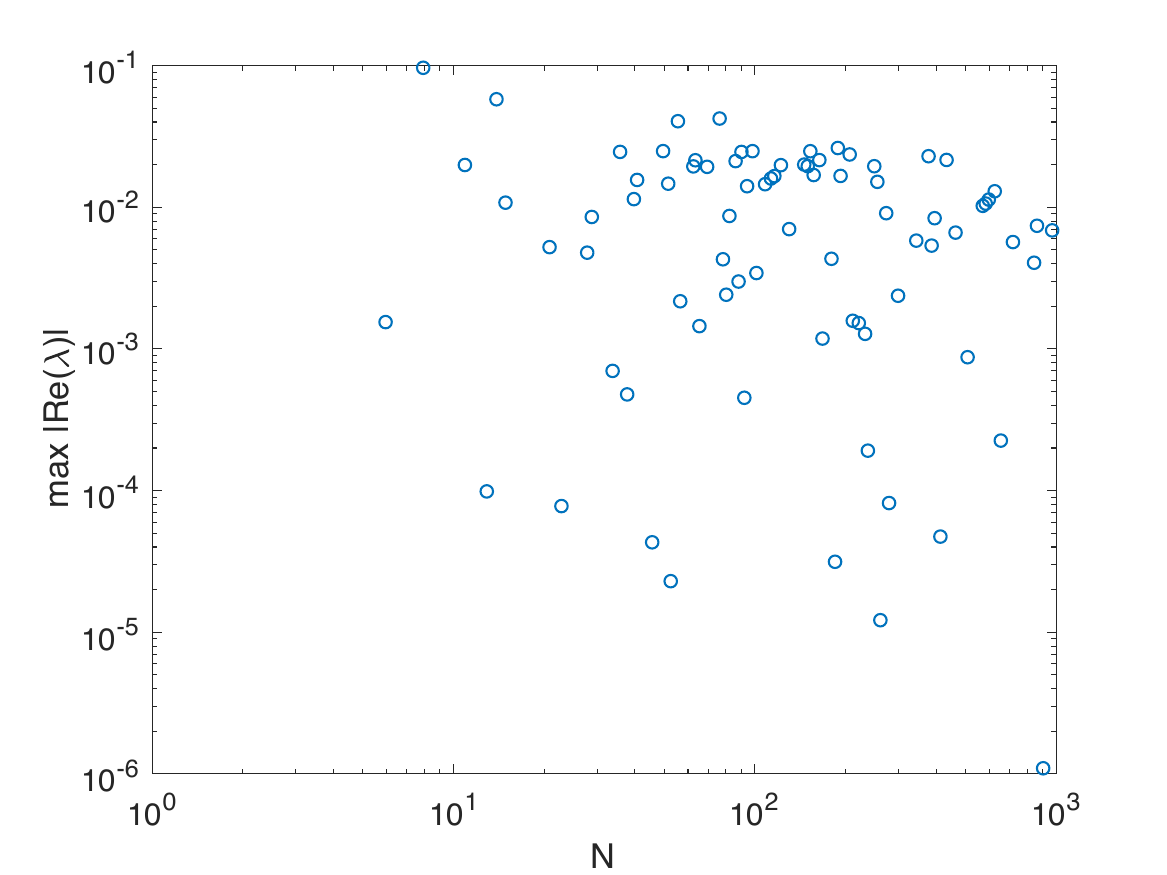}
\caption{Maximum of $|\RE \lambda|$, over all non-real eigenvalues $\lambda$ of the polynomial \eqref{Polyform}.}\label{F6}
\end{figure}
\end{example}
We present yet another possible application of the main results. Again, investigating this particular polynomial is motivated by examples from
\cite{BHMST}.

\begin{theorem}\label{2order}
Let
$$
C(z):=  z^2 \matp{c_1&&\\ &\ddots\\ && c_n} + \matp{d_1&&\\ &\ddots\\ && d_n}\in\Comp^{n\times n}[z],
$$ $$A_N(z):=C(z)\oplus 0_{N-n} \in\Comp^{N\times N}[z],
$$
be deterministic  matrix polynomials, where $n\geq1$ and  $c_1,\dots,c_n,d_1,\dots,d_n\in\Comp$ are fixed.
Let  $X_N(z)=z^2I+zX_N \in\Comp^{N\times N}[z]$ be a random matrix polynomial. We assume that

 \begin{itemize}
 \item[(a1.4)] $X_N$ is either a  Wigner matrix from Example \ref{Wigner} or a random sample covariance matrix from Example \ref{MP}, so that the resolvent of $z^2I_N+zX_N$ has a  limit law $z^{-1}m_{\x}(z) I_N$ on the family of sets $\mathbf S^{\x}_{N,\omega}$ with the rate $\Psi^{\x}_N(z)$, where $\x\in\set{\W,\MP}$, respectively, and let $\mathbf{T}$ be a compact set that does not intersect the real line.
\item[(a2.4)] $z^2c_i + d_i\neq 0$ for $z\in\mathbf S^{\x}_{N,\omega}$ or $z\in\mathbf{T}$, respectively, $i=1,\dots,n$, e.g. $c_i,d_i>0,$ for $i=1,\dots,n$.
 \end{itemize}

Then the eigenvalues of
$
X_N(z)+A_N(z)
$
are with high probability outside the set
$$
\widetilde{\mathbf{S}}_{N,\omega}^{\x}=\set{z\in\mathbf{S}_{N,\omega}^{\x}: \min_{i=1,\dots,n}  \left|   \frac1{z^2c_i+d_i}+z^{-1}m_{\x}(z)    \right|  >N^{-\beta\omega}}.
$$
and
\begin{equation}\label{Tmatrix1}
\widetilde{\mathbf{T}}_N:=  \set{z\in\mathbf  T : \min_{i=1,\dots,n}  \left|   \frac1{z^2c_i+d_i}+z^{-1}m_{\x}(z)    \right|    > N^{-\beta} } ,\end{equation}
 where $\beta<\alpha<\frac12$.
The resolvent of the polynomial $X_N(z)+A_N(z)$ has on $\widetilde{\mathbf{ S}}^{\x}_{N,\omega}$ and $\widetilde{\mathbf T}_N$ the following limit law
$$
\widetilde{M}_N(z)=z^{-1}m_{\x}(z)I_N -\matp{f_1&&\\ &\ddots\\ && f_n} ,
$$
$$
f_i=\frac{z^{-2} m_{\x}(z)^2}{(z^2c_i+d_i)^{-1}+z^{-1}m_{\x}(z)},\quad i=1,\dots n,
$$
with the rates
$$
N^\frac\omega2 |z|^{-2} \Psi^{\x}_N(z)\quad \text{ and } \quad   N^\alpha \Psi^{\x}_N(z),
$$
respectively.

Furthermore, for each solution $z_0$ with $z_0\in\Comp\setminus\Real$ of the equation $(z^2c_i+d_i)^{-1}+z^{-1}m_{\x}(z)  =0$ there exists  eigenvalues $\lambda^N_j$, $j=1,\dots,k$, where $k\geq 1$, of $X_N(z)$ converging to $z_0$ as
$$
|z_0-\lambda_j^N|\prec N^{-1/2},\quad j=1,\dots, k.
$$
\end{theorem}

\begin{proof} First note that the resolvent of $z^2I_N+zX_N$ has indeed the limit law $z^{-1}m_{\x}(z)I_N$ on the same sets and with the same convergence rate as  $X_N-zI_N$. Now the proof becomes another application of Theorem \ref{thres} with
$$
K(z)=\matp{ (z^2c_1+d_1)^{-1}+z^{-1}m_{\x}(z)\\ &\ddots\\ && (z^2c_n+d_n)^{-1}+z^{-1}m_{\x}(z)    }  ,
$$
 see  Theorem \ref{coruuu} for details.
We highlight that  the factor
$\norm{M(z)}_2^2$ in the formula for $\widetilde{\Psi}(z)$ in Theorem \ref{thres} cannot be ignored here (in previous applications introducing this factor was not necessary as $m_{\x}(z)$ is bounded in the upper half plane). However, in the current situation we have $\norm{M_N(z)}_2^2=\const |z|^{-2}$ in the case $\mathbf{S}_N^{\x}$ and $\norm{M_N(z)}_2^2=\const$ in the case $\mathbf{T}$.
\end{proof}

\color{black}

\section*{Acknowledgment}
The authors are indebted to Antti Knowles and Volker Mehrmann for inspiring discussions and helpful remarks.
The referee's remarks led to an essentially improved version of the manuscript, for which we express our gratitude.

\end{document}